\documentclass[aop, preprint, 12pt]{imsart}

\RequirePackage[OT1]{fontenc}
\RequirePackage{amsthm,amsmath}
\RequirePackage[numbers]{natbib}
\RequirePackage[colorlinks,citecolor=blue,urlcolor=blue]{hyperref}
\usepackage{amsfonts}
\usepackage{color}
\usepackage{amssymb}
\usepackage[english]{babel}
\usepackage[T1]{fontenc}
\usepackage{graphicx}
\usepackage{subfigure}
\usepackage{latexsym}
\usepackage{amstext}
\usepackage{mathrsfs}
\usepackage{hyperref}
\usepackage{dsfont}
\usepackage{graphics}
\usepackage{enumerate}
\usepackage{paralist}
\usepackage{color}
\usepackage{fullpage}
\newtheorem{prop}{Proposition}[section]

\newtheorem{rem}[prop]{Remark}

\newtheorem{cor}[prop]{Corollary}

\numberwithin{equation}{section}

\newcommand{\beq}{\begin{eqnarray}}
\newcommand{\beqq}{\begin{eqnarray*}}
\newcommand{\eeq}{\end{eqnarray}}
\newcommand{\eeqq}{\end{eqnarray*}}

\usepackage{mathtools}
\usepackage[english]{babel}
\usepackage[utf8]{inputenc}
\usepackage{amsmath}
\usepackage{graphicx}
\usepackage[colorinlistoftodos]{todonotes}
\usepackage{polynom}
\usepackage{amsfonts}
\usepackage{dsfont}
\usepackage{amsthm}
\usepackage{hyperref}
\usepackage{graphicx}
\newtheorem{theorem}{Theorem}[section]

\newtheorem{lemma}{Lemma}[section]
\newtheorem{proposition}[theorem]{Proposition}

\usepackage{fullpage}
\usepackage[T1]{fontenc}
\usepackage{hyperref}
\usepackage{xcolor}
\definecolor{link-color}{rgb}{0.15,0.4,0.15}
\hypersetup{
  colorlinks,
  linkcolor = {link-color}, citecolor = {link-color},
}

\usepackage{graphicx}
\usepackage{hyperref}
\usepackage{amsmath,amsthm,amssymb}
\usepackage{bbm}
\usepackage{tabularx}

\DeclareMathOperator{\rRe}{Re}

\renewcommand{\Re}{\rRe}

\newenvironment{eqnarr}{\begin{IEEEeqnarray}{rCl}}{\end{IEEEeqnarray}\ignorespacesafterend}

\renewcommand{\eqref}[1]{\hyperref[#1]{(\ref*{#1})}}

    \def\beq{\begin{eqnarr}}
    \def\eeq{\end{eqnarr}}
    \def\beqq{\begin{eqnarray*}} 
    \def\eeqq{\end{eqnarray*}} 

        \def\d{{\rm d}}

    \def\d{{\textnormal d}}

\newcommand*{\pref}[1]{\hyperref[#1]{(\ref*{#1})}}
\newcommand*{\refpref}[2]{\hyperref[#2]{\ref*{#1}(\ref*{#2})}}

\startlocaldefs
\numberwithin{equation}{section}
\theoremstyle{plain}

\endlocaldefs

\begin{document}

\begin{frontmatter}
\title{Deep factorisation of the stable  process III: Radial \\ excursion theory and the point of closest reach}

\runtitle{Deep factorisation of the stable process III}

\begin{aug}
%

\author{\fnms Andreas E. Kyprianou\thanksref{t1}\ead[label=e1]{a.kyprianou@bath.ac.uk, ws250@bath.ac.uk}},
\author{\fnms Victor Rivero\thanksref{t2}\ead[label=e2]{rivero@cimat.mx}}
\and
\author{\fnms{Weerapat Satitkanitkul}\ead[label=e1]{a.kyprianou@bath.ac.uk, ws250@bath.ac.uk}}

\ead[label=e3]{third@somewhere.com}

\thankstext{t2}{Supported by EPSRC grant EP/M001784/1}

\thankstext{t1}{Supported by EPSRC grant EP/L002442/1 and  	EP/M001784/1}


\affiliation{University of Bath, CIMAT A.C. and University of Bath}

\address{
Andreas Kyprianou and Weerapat Satitkanitkul,\\
Department of Mathematical Sciences \\
University of Bath\\
 Claverton Down\\ Bath, BA2 7AY\\
 UK.\\
\printead{e1}
}

\address{
Victor Rivero,\\
CIMAT A. C.,\\
Calle Jalisco s/n,\\
Col. Valenciana,\\
A. P. 402, C.P. 36000,\\
Guanajuato, Gto.,\\
Mexico.\\
\printead{e2}
}
\end{aug}

\begin{abstract}\hspace{0.1cm}
In this paper, we continue our understanding of the stable process from the perspective of the theory of self-similar Markov processes in the spirit of \cite{Deep1, Deep2}. In particular,  we turn our attention to the case of $d$-dimensional isotropic stable process, for $d\geq 2$. Using a completely new approach we consider the distribution of the point of closest reach. This leads us to a number of other substantial new results for this class of stable processes. We engage with a new radial excursion theory, never before used, from which we develop the classical Blumenthal--Getoor--Ray identities for first entry/exit into a ball, cf. \cite{BGR},   to the setting of $n$-tuple laws. We identify explicitly the stationary distribution of the stable process when reflected in its running radial supremum. Moreover, we provide  a representation of the Wiener--Hopf factorisation of the MAP that underlies the stable process through the Lamperti--Kiu transform. 
 \end{abstract}



\end{frontmatter}
\section{Introduction and main results}
For $d\geq 1$, let $X:= (X_t:t\geq 0)$, with probabilities $\mathbb{P}_x$, $x\in\mathbb{R}^d$,  be a $d$-dimensional isotropic stable process  of index $\alpha\in(0,2)$. That is to say that $X$ is a $\mathbb{R}^d$-valued L\'evy process having  characteristic triplet $(0,0,\Pi)$, where
\begin{equation}\label{jumpmeasure}
\Pi(B)=\frac{2^{\alpha}\Gamma(({d+\alpha})/{2})}{\pi^{d/2}|\Gamma(-{\alpha}/{2})|}\int_{B}\frac{1}{|y|^{\alpha+d}}{\rm d}y , \qquad B\in\mathcal{B}(\mathbb{R}).
\end{equation}
Equivalently, this means $X$ is a $d$-dimensional  L\'evy process with characteristic exponent $\Psi(\theta) = -\log\mathbb{E}_0({\rm e}^{{\rm i}\theta X_1})$ which satisfies
\[
\Psi(\theta) = |\theta|^\alpha, \qquad \theta\in\mathbb{R}.
\]

Stable processes are also self-similar in the sense that they satisfy a scaling property. More precisely, for $c>0$ and $x\in\mathbb{R}^d\setminus\{0\}$,
\begin{equation}\text{under } \mathbb{P}_x, \text{ the law of }(cX_{c^{-\alpha}t}, t\geq 0) \text{ is equal to } \mathbb{P}_{cx}.  \label{1/a}\end{equation}
As such, stable processes are useful prototypes for the study of  the class of L\'evy processes and, more recently, for the study of the class of self-similar Markov processes. The latter class of processes are regular strong Markov processes which respect the scaling relation \eqref{1/a}, and accordingly are identified as having stability index $1/\alpha$.

In the last few years, the fluctuation theory of one-dimensional stable processes has benefitted from the interplay between these two theories, in particular, exploiting  Lamperti-type decompositions of self-similar Markov processes. Examples of recent results include a deeper examination of the first passage problem, for the half-line, in one dimension,  \cite{n-tuple}, the distribution of the first point of entry into a strip, \cite{KPW}, and the stationary distribution of the process reflected in its radial maximum, \cite{Deep2}.

In this paper, we aim to push this agenda further into the setting of isotropic stable processes in dimension $d\geq 2$ (henceforth assumed).  Such processes are transient in the sense that  
\begin{equation}
\lim_{t\to\infty}|X_t|=\infty
\label{transient}
\end{equation}
almost surely. Accordingly, when issued from a point $x\neq 0$, it makes sense to define the point of closest reach to the origin; that is, the coordinates of the point in the closure of the range of $X$ with minimal radial distance from the origin. Our main results offer the exact distribution for the point of closest reach as well as a number of completely new fluctuation identities that fall out of its proof {\color{black}and the use of radial excursion theory}. 

Before describing them in more detail, let us define point of closest reach with a little more precision.  We need to note a number of facts. First, isotropy and transience ensures that $|X|$ is a conservative positive self-similar Markov process with index of self-similarity $1/\alpha$. Accordingly it can be represented via the classical Lamperti transformation
 \begin{equation}
 \label{lamperti}
|X_t| = {\rm e}^{\xi_{\varphi(t)}},\qquad  t\geq 0,
 \end{equation}
 where 
 \begin{equation}
 \label{varphi}
 \varphi(t) = \inf\{s>0: \int_0^s{\rm e}^{\alpha\xi_u}\d u>t\}
 \end{equation}
  and $\xi=(\xi_s: s\geq 0)$, with probabilities $\mathbf{P}_x$, $x\in\mathbb{R}$, is a L\'evy process. It was shown in \cite{CPP11} that the process $\xi$ belongs to the class of so-called hypergeometric L\'evy processes. In particular, its Wiener--Hopf factorisation is explicit. Indeed, suppose we write its characteristic exponent $\Psi_\xi(\theta) = -\log\mathbf{E}_0[\exp\{{\rm i}\theta\xi_1\}]$, $\theta\in\mathbb{R}$, then up to a multiplicative constant, 
 \begin{equation}
\Psi_\xi(\theta) = \frac{\Gamma(\frac{1}{2}(-{\rm i}\theta +\alpha ))}{\Gamma(-\frac{1}{2}{\rm i}\theta)}
\times
\frac{\Gamma(\frac{1}{2}({\rm i}\theta +d))}{\Gamma(\frac{1}{2}({\rm i}\theta +d-\alpha))}, \qquad \theta\in\mathbb{R},
\label{a}
\end{equation}
where the two terms either side of the multiplication sign constitute the two Wiener--Hopf factors. {\color{black} See e.g. Chapter VI in \cite{bertoin} for background.}
{\color{black}Recall that if $\Psi$ is the characteristic exponent of any L\'evy process, then there exist two Bernstein functions $\kappa$ and $\hat\kappa$ (see \cite{MR2978140} for a definition) such that, up to a multiplicative constant,
\begin{equation}\label{eq:normal_WH}
  \Psi({\rm i}\theta)=\kappa(-{\rm i}\theta)\hat\kappa({\rm i}\theta), \qquad \theta\in\mathbb{R}.
\end{equation}
Identity \eqref{eq:normal_WH} is what we refer to as the Wiener--Hopf factorisation.}
The left-hand factor codes the range of the running maximum  and the right-hand factor codes the range of the running infimum of $\xi$. It can be checked that both belong to the class of so-called beta subordinators (see \cite{KKP}, as well as some of the discussion later in this paper) and, in particular, have infinite activity. This {\color{black}implies} that $\xi$ is regular for both the upper and lower half-lines, which in turn, means that any sphere of radius $r>0$ is regular for both its interior and exterior for $X$. This and {\color{black}the fact that $X$ has} c\`adl\`ag paths ensures that, {\color{black} denoting} 
\[
{\texttt G}(t):= \sup\{s\leq t: |X_s| = \inf_{u\leq s}|X_u|\}, \qquad t\geq 0,
\]
the quantity $X_{{\texttt G}(t)}$ is well defined as the point of closest reach to the origin up to time $t$ in the sense that $X_{{\texttt G}(t) -} = X_{{\texttt G}(t)}$ and 
\[
|X_{{\texttt G}(t)} |= \inf_{s\leq t}|X_s|.
\]
The process $({\texttt G}(t), t\geq0)$ is monotone  increasing and hence there is no problem defining ${\texttt G}(\infty) = \lim_{t\to\infty}{\texttt G}(t)$ almost surely. Moreover, as $X$ is transient in the sense of \eqref{transient},  it is also clear that, almost surely, ${\texttt G}(\infty) = {\texttt G}(t)$ for all $t$ sufficiently large and that 
\[
|X_{{\texttt G}(\infty)} |= \inf_{s\geq 0}|X_s|.
\]
{\color{black} Our first main result provides explicitly the law of $X_{{\texttt G}(\infty)}.$}
\begin{theorem}[Point of Closest Reach to the origin]\label{main1} The law of the point of closest reach to the origin is given by 
\[
\mathbb{P}_x(X_{\emph{\texttt G}(\infty)}\in \d  y)=\pi^{-d/2}\dfrac{\Gamma\left({d}/{2}\right)^2}{\Gamma\left(({d-\alpha})/{2}\right)\Gamma\left({\alpha}/{2}\right)}
\, \frac{(|x|^2-|y|^2)^{\alpha/2}}{|x-y|^{d}|y|^{\alpha}}{\rm d} y,\qquad 0<|y|<|x|.
\]
\end{theorem}

Fundamentally, the proof of Theorem \ref{main1} will be derived from two main facts. The first is a suite of   exit/entrance formulae from balls for stable processes  which come from the classical work of Blumenthal--Getoor--Ray \cite{BGR}. To state these results, let us write 
\[
\tau^\oplus_r = \inf\{t>0: |X_t|<r\}\text{ and } \tau^\ominus_r = \inf\{t>0: |X_t|>r\},
\]
for $r>0$.
\begin{theorem}[Blumenthal--Getoor--Ray \cite{BGR}]\label{BGRtheorem}
For either $|x|<r<|y|$ when $\tau = \tau^\ominus_r$, or $|y|<r<|x|$ when $\tau = \tau^\oplus_r$,
\begin{equation}
\label{BGR}
\mathbb{P}_x(X_{\tau}\in {\rm d} y)=\pi^{-(d/2+1)}\Gamma\left({d}/{2} \right)\sin\left(\frac{\pi\alpha}{2}\right) \frac{|r^2-|x|^2|^{\alpha/2}}{|r^2-|y|^2|^{\alpha/2}}|x-y|^{-d}{\rm d} y.
\end{equation}
Moreover, for $|x|>r$, 
\begin{equation}
\mathbb{P}_x(\tau^\oplus_r = \infty) =  \frac{\Gamma(d/2)}{\Gamma((d - \alpha)/2)\Gamma(\alpha/2)}
\int_0^{(|x|^2/r^2)-1} (u+1)^{-d/2}u^{\alpha/2-1}\d u
\label{BGR2}
\end{equation}
and, for $|x|<r$ and bounded measurable $f$ on $\mathbb{R}^d$,
\[
\mathbb{E}_x\left[\int_0^{\tau^\ominus_r} f(X_s)\d s\right] = \int_{|y|>r}h^\ominus_r(x, y)f(y)\d y
\]
such that 
\begin{equation}
h^\ominus_r(x, y) = 2^{-\alpha}\pi^{-d/2}\frac{\Gamma(d/2)}{\Gamma(\alpha/2)^2}|x-y|^{\alpha - d} \int_0^{\zeta_r(x,y)}  (u+1)^{-d/2}u^{\alpha/2-1}\d u, \qquad |y|<r,
\label{BGR3}
\end{equation}
where $ \zeta^\ominus_r(x,y)= (r^2-|x|^2)(r^2-|y|^2)/r^2|x-y|^2$.
\end{theorem}

\begin{rem}\rm
\label{rem1}
It is worth remarking that \eqref{BGR2} can be used to derive the density of $|X_{{\texttt G}(\infty)}|$ quite easily.  {\color{black} Indeed, thanks to the scaling property} and rotational symmetry, it suffices in this respect to consider the law of $|X_{{\texttt G}(\infty)}|$ under $\mathbb{P}_{\texttt{1}}$, where ${\texttt 1} = (1,0,\cdots,0)$ is the `North Pole' on $\mathbb{S}_{d-1}$.
In this respect, we note that $\mathbb{P}_{\texttt{1}}(|X_{{\texttt G}(\infty)}| \leq r) =1- \mathbb{P}_{\texttt{1}}(\tau^\oplus_r = \infty)$, hence, {\color{black}for $\gamma>0$},
\begin{align*}
\mathbb{E}_{\texttt{1}}[|X_{{\texttt G}(\infty)}|^{2\gamma}]& =\int_0^1 r^{2\gamma} {\d}\mathbb{P}_{\texttt{1}}(|X_{{\texttt G}(\infty)}| \leq r)\\
&=\frac{2\Gamma(d/2)}{\Gamma((d - \alpha)/2)\Gamma(\alpha/2)}\int_0^1r^{2\gamma+ (d-\alpha)-1}(1-r^2)^{\frac{\alpha}{2} - 1}\d r\\
& = \frac{\Gamma(d/2)}{\Gamma((d - \alpha)/2)\Gamma(\alpha/2)}\int_0^1 u^{\gamma+ \frac{(d-\alpha)}{2} -1}(1-u)^{\frac{\alpha}{2}- 1}\d u.
\end{align*}
From this it is straightforward to see that $|X_{{\texttt G}(\infty)}|$ under $\mathbb{P}_{\texttt{1}}$ is equal in law to $\sqrt{\texttt{A} }$, where $\texttt{A}$ is a Beta$((d-\alpha)/2, \alpha/2)$ distribution. 
 \end{rem}

The second main fact that drives the proof of Theorem \ref{main1} is the Lamperti--Kiu representation of self-similar Markov processes. To describe it, we need to introduce the notion of a Markov Additive Process, henceforth written MAP for short.

Let $\mathbb{S}_{d-1} =\{x\in\mathbb{R}^d: |x|=1\}$. With an abuse  of previous notation, we say that $(\xi,\Theta) =( (\xi_t, \Theta_t), t\geq 0)$ is a MAP if it is a regular Strong Markov Process on $\mathbb{R}\times\mathbb{S}_{d-1},$ with probabilities $\mathbf{P}_{x,\theta}$, $x\in\mathbb{R}^d$, $\theta\in\mathbb{S}_{d-1}$,  such that, for any $t\geq 0$, the conditional law of the process $((\xi_{s+t}-\xi_t,\Theta_{s+t}):s\geq 0)$, given  $\{(\xi_u,\Theta_u), u\leq t\},$ is that of $(\xi,\Theta)$ under $\mathbf{P}_{0,\theta}$, with $\theta=\Theta_t$. For a MAP pair $(\xi, \Theta)$, we call $\xi$ the {\it ordinate} and $\Theta$ the {\it modulator}.
 
According to one of the main results in \cite{ACGZ}, there exists a MAP such that the $d$-dimensional isotropic stable process can be written
   \begin{equation}\label{eq:lamperti_kiu}
    X_t  = \exp\{\xi_{\varphi(t)}\} \Theta_{\varphi(t)}\qquad t \geq 0,
  \end{equation}
  where $\varphi$ has the same definition as \eqref{varphi}.
Now we see the reason for our preemptive choice of notation as clearly $|X_t|$ now agrees with \eqref{lamperti} and we can understand e.g. $\mathbf{P}_x(\xi_t \in A) = \int_{\mathbb{S}_{d-1}}\mathbf{P}_{x,\theta}(\xi_t\in A, \, \Theta_t\in \d\theta)$, for $t\geq 0$ and $A\in\mathcal{B}(\mathbb{R})$. Whilst the processes $\Theta$ and $\xi$ are corollated, it is clearly the case that $\Theta$ is isotropic in the distributional sense, and hence an ergodic process on a compact domain with uniform stationary distribution.

\begin{rem}\rm
Noting that  $X_{{\texttt G}(\infty)} = |X_{{\texttt G}(\infty)}|\times \arg(X_{{\texttt G}(\infty)})$, it is tempting to believe that it is a simple step to take the distributional identity in Remark \ref{rem1} into the law of $X_{{\texttt G}(\infty)}$. Somewhat naively, this is a particularly attractive perspective because of the similarity between \eqref{BGR} and the {\it a postiori} conclusion in Theorem \ref{main1}.  Indeed one of our approaches was to try to derive the one from the other by a simple limiting procedure. Making this idea rigorous  turned out to be much more difficult than originally anticipated on account of the very subtle nature of the correlation between radial and angular behaviour of the MAP that underlies the stable process. 
\end{rem}

Our proof of Theorem \ref{main1} will take us on a journey through an excursion theory of $X$ from its radial maximum. In dimension $d\geq 2,$ this is the first time, to our knowledge, that such a radial excursion theory has been used, see however \cite{CKPR12}.  This will also allow us to prove the $n$-tuple laws at first entry/exit of a ball (below), which provide a non-trivial extension to the classical identities of Blumenthal, Getoor and Ray \cite{BGR} given in Theorem \ref{BGRtheorem}.
Indeed, once the relevant radial excursion theory is made clear, the following theorem and its corollary emerge as a consequence of an application of the appropriate {\color{black}exit} system, very much in the spirit of how analogous calculations would be made e.g. in the setting of L\'evy processes. What makes them difficult, however, is that the underlying excursion theory deals with excursions of the process $X_t/M_t$, $t\geq 0$, away from the {\it set} $\mathbb{S}_{d-1}$, where $M_t := \sup_{s\leq t} |X_{s}|$, $t\geq 0$. As such it is significantly harder to deal with the family of  associated excursion measures that appear in the {\color{black}exit} system and which are indexed by $\mathbb{S}_{d-1}${\color{black}, see below for further details}.

\begin{theorem}[Triple law at first entrance/exit of a ball]\label{tripleintheorem}
Fix $r>0$ and define, for $x,z,y,v\in\mathbb{R}^d\backslash\{0\}$,
\[
\chi_x(z, y, v): = \pi^{-3d/2}\frac{\Gamma(({d+\alpha})/{2})}{|\Gamma(-{\alpha}/{2})|}\frac{\Gamma(d/2)^2}{\Gamma(\alpha/2)^2} \frac{||z|^2-|x|^2|^{\alpha/2}||y|^2-|z|^2|^{\alpha/2}}
{|z|^\alpha|z-x|^{d}|z-y|^{d}|v-y|^{\alpha+d}} .
\] 
\begin{itemize}
\item[(i)] Write 
\[
{\emph{\texttt G}}(\tau_r^\oplus) = \sup\{s<\tau^\oplus_r: |X_s| =\inf_{u\leq s}|X_u| \}
\]
for the instant of closest reach of the origin before first entry into $r\mathbb{S}_{d-1}$.
 For $|x|>|z|>r$, $|y|>|z|$ and $|v|<r$,
\begin{align*}
&\mathbb{P}_x(X_{{\emph{\texttt G}}(\tau_r^\oplus)}\in \d z, \, X_{\tau^\oplus_r-}\in \d y,\, X_{\tau^\oplus_r}\in \d v;\, \tau^\oplus_r<\infty)
=
 \chi_x(z, y, v)\,
\d z\, \d y\,\d v.
\end{align*}
\item[(ii)] Define $\mathcal{G}(t) = \sup\{s<t: |X_s| = \sup_{u\leq s}|X_u|\}$, $t\geq 0$,
and write
\[
\mathcal{G}(\tau^\ominus_r) = \sup\{s<\tau^\ominus_r: |X_s| =\sup_{u\leq s}|X_u| \}.
\]
for the instant of furtherest reach from the origin immediately before first exit from $r\mathbb{S}_{d-1}$. For $|x|<|z|<r$, $|y|<|z|$ and $|v|>r$,
\[
\mathbb{P}_x(X_{{\mathcal G}(\tau_r^\ominus)}\in \d z, \, X_{\tau^\ominus_r-}\in \d y,\, X_{\tau^\ominus_r}\in \d v)=
 \chi_x(z, y, v)\,
\d z\, \d y\,\d v.
\]

\end{itemize}
\end{theorem}

Marginalising the first  triple law in Theorem \ref{main1} to give  the joint law of the pair   $(X_{{\texttt G}(\tau_r^\oplus)},\, X_{\tau^\oplus_r})$ or the pair $(X_{\tau^\oplus_r-},\, X_{\tau^\oplus_r})$ is not necessarily straightforward (although the reader familiar with the manipulation of Riesz potentials may feel more comfortable as such). 
Whist an analytical computation for the marginalisation should be possible, if not tedious, we provide a proof which   combines other fluctuation identities that we will uncover {\it en route}.  

\begin{cor}[First entrance/exit and closest reach]\label{double1}
Fix $r>0$ and define, for $x,z,v\in\mathbb{R}^d\backslash\{0\}$,
\[
\chi_x(z, \bullet, v): = \frac{\Gamma(d/2)^2}{\pi^{d}|\Gamma(-\alpha/2)|\Gamma(\alpha/2)} 
\frac{||z|^2-|x|^2|^{\alpha/2}}{||z|^2-|v|^2|^{\alpha/2}|z-v|^{d}|z-x|^{d}}.
\] 
\begin{itemize}
\item[(i)] 
 For $|x|>|z|>r$,  $|v|<r$,
\begin{align*}
&\mathbb{P}_x(X_{{\emph{\texttt G}}(\tau_r^\oplus)}\in \d z, \,X_{\tau^\oplus_r}\in \d v;\, \tau^\oplus_r<\infty)
=
 \chi_x(z, \bullet, v)
\d z\, \d v.
\end{align*}
\item[(ii)]   For $|x|<|z|<r$ and $|v|>r$,
\[
\mathbb{P}_x(X_{{{\mathcal G}}(\tau_r^\ominus)}\in \d z,  \,X_{\tau^\ominus_r}\in \d v)=
 \chi_x(z, \bullet, v)\,
\d z\, \d v.
\]
\end{itemize}
\end{cor}

\begin{cor}[First entrance/exit and preceding position]\label{double2}
Fix $r>0$ and define, for $x,z,y,v\in\mathbb{R}^d\backslash\{0\}$,
\[
\chi_x(\bullet, y, v): = \frac{\Gamma(({d+\alpha})/{2})\Gamma(d/2)}{\pi^{d}|\Gamma(-{\alpha}/{2})|\Gamma(\alpha/2)^2}
\left( \int_0^{\zeta^\oplus_r(x,y)}  (u+1)^{-d/2}u^{\alpha/2-1}\d u\right)
 \frac{|x-y|^{\alpha - d}}{|v-y|^{\alpha + d}} \, \d v \,\d y,
\] 
where 
\[
\zeta^\oplus_r(x,y):= (|x|^2-r^2)(|y|^2-r^2)/r^2|x-y|^2.
\]
\begin{itemize}
\item[(i)] 
 For $|x|, |y|>r$,  $|v|<r$,
\begin{align*}
&\mathbb{P}_x( X_{\tau^\oplus_r-}\in \d y, \,X_{\tau^\oplus_r}\in \d v;\, \tau^\oplus_r<\infty)
=
 \chi_x(\bullet, y, v)
\d y\, \d v.
\end{align*}
\item[(ii)]   For $|x|, |y|<r$ and $|v|>r$,
\[
\mathbb{P}_x( X_{\tau^\ominus_r-}\in \d y, \, X_{\tau^\ominus_r}\in \d v)=
 \chi_x(\bullet, y, v)\,
\d y\, \d v.
\]
\end{itemize}
\end{cor}


  In \cite{Deep1, Deep2}, one-dimensional stable processes were considered (up to first hitting of the origin in the case that $\alpha\in(1,2)$), for which the process $\Theta$ in the underlying MAP  is nothing more than a two-state Markov chain on $\{1,-1\}$. Such MAPs are known to have a Wiener--Hopf-type decomposition. 
 
 To be more precise, one may describe the semigroup of $(\xi,\Theta)$ via a matrix Laplace exponent which plays a similar role to the characteristic exponent of $\xi$. When it exists, the matrix $\boldsymbol \Psi$, mapping $\mathbb C$ to the space of $2\times 2$ complex valued matrices\footnote{Here  the matrix entries are arranged by
    \[
      A=\left(\begin{matrix}
          A_{1,1} & A_{1,-1}\\
          A_{-1,1} & A_{-1,-1}
      \end{matrix}\right).
    \]
  }, satisfies,
  \[
    (e^{-\boldsymbol \Psi(z) t})_{i,j}=\mathbf{E}_{0,i}[e^{-z\xi(t)};J_t=j], \qquad i,j = \pm1, t\geq 0.
  \]
In fact, it is known to take the form 
\begin{equation}
  \boldsymbol{\Psi}(z) =\left(
  \begin{array}{cc}
     \dfrac{\Gamma(\alpha+z)\Gamma(1-z)}
      {\Gamma(\alpha\hat\rho+z)\Gamma(1-\alpha\hat\rho- z)}
    & -\dfrac{\Gamma(\alpha+z)\Gamma(1-z)}
      {\Gamma(\alpha\hat\rho)\Gamma(1-\alpha\hat\rho)}
    \\
    &\\
  -  \dfrac{\Gamma(\alpha+z)\Gamma(1-z)}
      {\Gamma(\alpha\rho)\Gamma(1-\alpha\rho)}
    &  \dfrac{\Gamma(\alpha+z)\Gamma(1-z)}
      {\Gamma(\alpha\rho+z)\Gamma(1-\alpha\rho-z)}
  \end{array} 
  \right),
  \label{MAPHG}
\end{equation}
  for Re$({z}) \in (-1,\alpha)$; see \cite{CPR} and \cite{KKPW}.   Similar to the case of L\'evy processes, we can define $\boldsymbol\kappa$ and $\hat{\boldsymbol\kappa}$ as the matrix Laplace exponents of two MAPs, each with non-decreasing ordinate, whose ordinate ranges and accompanying modulation coincide in distribution with the the range of the running maximum of $\xi$ and that of the dual process $\hat\xi$, with accompanying modulation. 
  The analogue of the Wiener--Hopf factorisation for MAPs states that, up to pre-multiplying $\boldsymbol \kappa$ or $\hat{\boldsymbol \kappa}$ (and hence equivalently up to pre-multiplying $\boldsymbol \Psi$) by a strictly positive diagonal matrix, we have that
  \begin{equation}\label{eq:factorisation}
    \boldsymbol \Psi(-{\rm i}\lambda) = {\boldsymbol\Delta}^{-1}_{\pi}\hat{\boldsymbol \kappa}({\rm i}\lambda)^T{\boldsymbol\Delta}_{\pi}\boldsymbol\kappa(-{\rm i}\lambda),
  \end{equation}
for $\lambda\in\mathbb{R}$,  where
  \begin{equation*}
    {\boldsymbol\Delta}_\pi:=
    \left(\begin{array}{cc}
        \sin(\pi\alpha\rho),& 0 \\
        0 &
         \sin(\pi\alpha\hat\rho)
    \end{array}\right).
      \end{equation*}

In the setting of the MAP which underlies the stable process, the so-called deep Wiener--Hopf factorisation was computed in \cite{Deep1}, thereby providing the first  explicit example of the Wiener--Hopf factorisation for a MAP. When $X$ is a symmetric one-dimensional stable process, then, without loss of generality, we may take ${\boldsymbol\Delta}_{\pi}$ as the identity matrix, the underlying MAP becomes  symmetric,  in which case $\hat{\boldsymbol \kappa}^T = \hat{\boldsymbol \kappa}$ and, moreover, $\hat{\boldsymbol \kappa}(\lambda) = {\boldsymbol \kappa}(\lambda+1-\alpha)$, $\lambda\geq0$. In that case, the factorisation  simplifies to 
\begin{equation}
 \boldsymbol \Psi(-{\rm i}\lambda) ={\boldsymbol \kappa}({\rm i}\lambda + 1-\alpha){\boldsymbol \kappa}(-{\rm i}\lambda), \qquad \lambda\in\mathbb{R},
\label{WHFsymmetric1D}
\end{equation}
up to multiplication by a strictly positive diagonal matrix.

For dimension $d\geq 2$, by adopting the right mathematical language, we are also able to provide the deep factorisation of the $d$-dimensional isotropic stable process, which also generalises the situation in one dimension. To this end, let us introduce the notion of the descending ladder MAP process for $(\xi,\Theta)$. 

It is not difficult to show that the pair $((\overline\xi_t- {\xi}_t, \Theta_t), t\geq 0)$, forms a strong Markov process, where $\overline\xi_t: = \sup_{s\leq t}\xi_s$, $t\geq 0$ is the running maximum of $\xi$. Naturally, on account of the fact that $\xi$, as a lone process, is a L\'evy process, $(\overline\xi_t- {\xi}_t,\ t\geq0)$,  is also a strong Markov process, but we are more interested here on its dependency on $\Theta$. If we denote by $L$ the local time at zero of $\bar\xi - \xi$, then the strong Markov property tells us that $(L^{-1}_t, H^+_t, \Theta^+_t)$, $t\geq 0$, defines a Markov additive process, whose first two elements are ordinates that are non-decreasing, where $H^+_t =\xi_{L^{-1}_t}$ and whose modulator $\Theta^+_t = \Theta_{L^{-1}_t}$, $t\geq 0$. In this sense, $L$ also serves as a local time on the set $\{0\}\times\mathbb{S}_{d-1}$ of the Markov process $(\overline\xi - \xi, \Theta)$. Because $\xi$, alone, is also a L\'evy process then the pair $(L^{-1}, H^+)$, without reference to the associated modulation $\Theta^+$, are Markovian and play the role of the ascending ladder time and height subordinators of $\xi$. But again, we are more concerned here with their dependency on $\Theta^+$.

If we are to state a factorisation analogous to \eqref{WHFsymmetric1D}, we must understand  how we should define the  quantities that are analogous to $\boldsymbol\Psi$ and $\boldsymbol\kappa$. Inspiration to this end comes from
  \cite{Deep2}, where  it was shown that it is more convenient to understand the relationship  \eqref{eq:factorisation} in its inverse form. This is equivalent to showing how the resolvent of the    underlying MAP relates to the  potential measures associated to $\boldsymbol\kappa$ and $\hat{\boldsymbol\kappa}$.

Therefore, in the current setting of $d$-dimensional isotropic stable processes, we define the operators 
\[
{\bf R}_z[f](\theta) = \mathbf{E}_{0,\theta}\left[\int_0^\infty {\rm e}^{-z \xi_t}f(\Theta_t)\d t\right], \qquad \theta\in\mathbb{S}_{d-1}, z\in\mathbb{C}
\]
and
\[
\boldsymbol\rho_z[f](\theta) = \mathbf{E}_{0,\theta}\left[\int_0^\infty {\rm e}^{-z H^+_t}f(\Theta^+_t)\d t\right], \qquad \theta\in\mathbb{S}_{d-1}, z\in\mathbb{C},
\]
for bounded measurable $f: \mathbb{S}_{d-1}\mapsto[0,\infty)$, whenever the integrals make sense.

\begin{theorem}[Deep factorisation of the $d$-dimensional isotropic stable process]\label{WHF} Suppose that  $f: \mathbb{S}_{d-1}\mapsto\mathbb{R}$ is bounded and measurable. 
Then
\[
{\bf R}_{-{\rm i}\lambda}[f](\theta) = C_{\alpha,d}\, \boldsymbol\rho_{{\rm i}\lambda + d-\alpha}\big[\boldsymbol\rho_{-{\rm i}\lambda}[f]\big](\theta) , \qquad \theta\in\mathbb{S}_{d-1},
\lambda\in\mathbb{R},
\]
where 
$
C_{\alpha, d} = 2^{-\alpha}{\Gamma((d - \alpha)/2)^2}/{\Gamma(d/2)^2}. 
$ Moreover, 
\[
\boldsymbol\rho_z[f](\theta) =\pi^{-d/2}\frac{\Gamma(d/2)^2}{\Gamma((d - \alpha)/2)\Gamma(\alpha/2)}\int_{|y|>1}f(\arg(y)) 
\frac{||y|^2-|\theta|^2|^{\alpha/2}}
{|y|^{\alpha+z}|\theta-y|^{d} }\d y
, \qquad \Re(z)\geq 0
\]
and 
\[
{\bf R}_{-{\rm i}\lambda}[f](\theta) = \frac{\Gamma((d-\alpha)/2)}{2^\alpha\pi^{d/2}\Gamma(\alpha/2)}
\int_{\mathbb{R}^d}f(\arg(y))|y-\theta|^{{\rm i}\lambda-d} \d y, \qquad \lambda\in\mathbb{R}.
\]
\end{theorem}

\noindent {\color{black}This, our third main result,} is the first example we know of in the literature which {\color{black}provides} in explicit detail the Wiener--Hopf factorisation of a MAP for which the modulator has an uncountable state space. 

\bigskip

Our final main result concerns the stationary distribution of the stable process reflected in its radial supremum. Define $M_t =\sup_{s\leq t}|X_s|$, $t\geq 0$. It is a straightforward computation to show that $(X_t/M_t, M_t)$, $t\geq0$ is a Markov process which lives on $\mathbb{B}_d\times(0,\infty)$, where $\mathbb{B}_d = \{x\in\mathbb{R}^d: |x|\leq 1\}$. Thanks to the transience of $X$, it is clear that $\lim_{t\to\infty}M_t = \infty$, however, thanks to repeated {\color{black}normalistaion} of $X$ by its radial maximum, we can expect that the $\lim_{t\to\infty}X_t/M_t$ exists in distribution. Indeed, in the one-dimensional setting this has already been proved to be the case in \cite{Deep2}.

\begin{theorem}\label{stdist}
For all bounded measurable $f: \mathbb{B}_d\mapsto \mathbb{R}$ and $x\in\mathbb{R}\backslash\{0\}$
\begin{align*}
&\lim_{t\to\infty}\mathbb{E}_x[f(X_t/M_t)] =\pi^{-d/2}\frac{\Gamma((d+\alpha)/2)}{\Gamma(\alpha/2)}
\int_{\mathbb{S}_{d-1}} \sigma_1(\d\phi)\int_{|w|<1} f(w) 
\frac{|1-|w|^2|^{\alpha/2}  }
{|\phi-w|^{d} }{\d w},
\end{align*}
where $\sigma_1(\d y)$ is the surface measure on $\mathbb{S}_{d-1}$, normalised to have unit mass.

\end{theorem}
\begin{rem}\rm
 Although we are dealing with the case $d\geq 2$,  with the help of the duplication formula for gamma functions, we can verify that the above limiting identity agrees with the stationary distribution for the radially reflected process  when $d=1$  given in  Theorem 1.3 in \cite{Deep2} if we set $d = 1$ and $\alpha\in(0,1)$.
 
 We also note that
the stationary distribution in the previous theorem is equal in law to $\texttt{U}\times\sqrt{\texttt{B}}$, where $\texttt{U}$ is uniformly distributed on $\mathbb{S}_{d-1}$ and  $\texttt{B}$ is a Beta$(d/2, {\alpha}/{2})$ distribution.
 {\color{black}if} Indeed, suppose we take $f(w)  = |w|^{2\gamma} g(\arg(w))$ for {\color{black}$\gamma>0$}, then we also see that 
 \[
 \lim_{t\to\infty}\mathbb{E}_x[ f(X_t/M_t)] =\frac{2\Gamma((d+\alpha)/2)}{\Gamma(d/2)\Gamma(\alpha/2)}
\int_{\mathbb{S}_{d-1}} \sigma_1(\d\phi)
\int_0^1 r^{2\gamma+d-1} (1-r^2)^{\alpha/2} \d r \int_{\mathbb{S}_{d-1}}
\frac{g(\theta)}
{|\phi-r\theta|^{d} } \sigma_1(\d\theta).
 \]
 A Newton potential formula  tells us that
  $
 \int_{\mathbb{S}_{d-1}}
{|\phi-r\theta|^{-d} }  \sigma_1(\d\phi)= 1
$, see for example Remark III.2.5 in \cite{ALEAKyp},
and hence, after an application of Fubini's theorem for the two spherical integrals and change of variable, 
\[
 \lim_{t\to\infty}\mathbb{E}_x[ f(X_t/M_t)] =\frac{\Gamma((d+\alpha)/2)}{\Gamma(d/2)\Gamma(\alpha/2)}
\int_0^1 u^{\gamma + \frac{d}{2}-1} (1-u)^{\frac{\alpha}{2}-1} \d u \times  \int_{\mathbb{S}_{d-1}} g(\theta) \sigma_1(\d\theta),
 \]
verifying the claimed distributional decomposition.
 \end{rem}

The remainder of this paper is structured as follows. In the next section we discuss the fundamental tool that allows us to conduct our analysis: an appropriate excursion theory of the underlying MAP $(\xi, \Theta)$. This may otherwise be understood as (up to a {\color{black} change of time and change of scale space}) the excursion of $X$ from its radial minimum.  With this in hand, we progress directly to the proof of Theorem \ref{main1} in Section \ref{proofofmain1}. Thereafter, in Section \ref{sectionRBZ}, we introduce the so-called Riesz--Bogdan--\.Zak transform and discuss its relation to some of the key quantities that appear in the aforesaid radial fluctuation theory.  Next, in Section \ref{IntegrationN} we analyse in more detail some specific identities pertaining to integration with respect to the excursion measure that appears in Section \ref{RET}. These identities are then used to prove Theorem \ref{tripleintheorem} in Section \ref{triplelaws} and to prove the deep factorisation in Section \ref{deepproof}. Finally, we deal with the stationary distribution, which is proved in Section \ref{8}.

\section{Radial excursion theory}\label{RET}

One of the principal tools that we will use in our computations is that of radial excursion theory of $X$ from its running minimum. In order to build such a theory, we return to the  Lamperti--Kiu transformation \eqref{eq:lamperti_kiu}.
In the spirit of the discussion preceding Theorem \ref{WHF}, by considering, say, $\ell = (\ell_t, t\geq 0)$, the local time at $0$ of the reflected L\'evy process $(\xi_t - \underline{\xi}_t,  t\geq 0)$, where $\underline{\xi}_t : = \inf_{s\leq t}\xi_s$, $t\geq0$, we can build the descending ladder MAP $((H^-_t, \Theta^-_t), t\geq 0)$, in the obvious way. As before, although the local time $\ell$ pertains to the reflected L\'evy process $\xi-\underline\xi$, we will see below that it serves as an adequate choice for the local time 
of the Markov process $(\xi-\underline\xi, \Theta)$ on the set $\{0\}\times\mathbb{S}_{d-1}$ to the extent that we can use it in the context of Maisonneuve's exit formula. 

More precisely, 
suppose we define ${\texttt g}_t =\sup\{s< t: \xi_s = \underline\xi_s\},$ and recall that the regularity of $\xi$ for $(-\infty, 0)$ and $(0,\infty)$ ensures that it is well defined, as is ${\texttt g}_\infty = \lim_{t\to\infty}{\texttt g}_t$. Set 
\[
{\texttt d}_t = \inf\{s> t: \xi_s = \underline\xi_s\}
\]
and, for all $t>0$ such that ${\texttt d}_t > {\texttt g}_t$ the process 
\[
(\epsilon_{{\texttt g}_t}(s), \Theta^\epsilon_{{\texttt g}_t}(s)): = (\xi_{{\texttt g}_t +s}-\xi_{{\texttt g}_t}, \Theta_{{\texttt{g}_t + s}}), \qquad s\leq \zeta_{{\texttt g}_t}: = {\texttt d}_t-{\texttt g}_t,
\]
codes the excursion of $(\xi-\underline\xi, \Theta)$ from the set $(0,\mathbb{S}_{d-1})$ which straddles time $t$. Such excursions live in the space of $\mathbb{U}(\mathbb{R}\times\mathbb{S}_{d-1})$, the space of c\`adl\`ag paths 
with lifetime  $\zeta = \inf\{s>0: \epsilon(s) <0\}$ such that $(\epsilon(0),\Theta^\epsilon(0))\in \{0\}\times\mathbb{S}_{d-1}$, $(\epsilon(s), \Theta^\epsilon(s))\in (0,\infty)\times\mathbb{S}_{d-1}$, for $0<s<\zeta$,  and $\epsilon(\zeta)\in (-\infty,0)$.

Taking account of the Lamperti--Kiu transform \eqref{eq:lamperti_kiu}, it is natural to consider how the excursion of $(\xi - \underline{\xi}, \Theta)$ from $\{0\}\times\mathbb{S}_{d-1}$ translates into a radial excursion theory for the process 
\[
Y_t  : = {\rm e}^{\xi_t}\Theta_t, \qquad t\geq 0.
\]
Ignoring the time change in \eqref{eq:lamperti_kiu}, we see that the radial minima of the process $Y$ agree with the radial minima of the stable process $X$.
Indeed, an excursion of $(\xi - \underline{\xi}, \Theta)$ from $\{0\}\times\mathbb{S}_{d-1}$ constitutes an excursion of $(Y_t/\inf_{s\leq t}|Y_s|, t\geq 0)$, from $\mathbb{S}_{d-1}$, or equivalently an excursion of $Y$ from its running radial infimum. Moreover, we see that, for  all $t>0$ such that ${\texttt d}_t > {\texttt g}_t$, 
\[
Y_{\texttt{g}_t + s} = {\rm e}^{\xi_{\texttt{g}_t}} {\rm e}^{\epsilon_{\texttt{g}_t}(s)} \Theta^\epsilon_{{\texttt g}_t}(s) = |Y_{\texttt{g}_t}|{\rm e}^{\epsilon_{\texttt{g}_t}(s)} \Theta^\epsilon_{{\texttt g}_t}(s)
, \qquad s\leq \zeta_{{\texttt g}_t}.
\]
This will be useful to keep in mind in the forthcoming excursion computations.

For $t>0,$ let $R_t=\texttt{d}_t-t,$ and define the set $G = \{t>0: R_{t-}=0, R_{t}>0\} = \{{\texttt g}_s: s\geq 0\}$. 
The classical theory of exit systems in  \cite{M75} now implies that there exists an additive functional $(\Lambda_t, t\geq 0)$ carried by the set of times $\{t\geq0: (\xi_t - \underline{\xi}_t, \Theta_t)\in\{0\}\times\mathbb{S}_{d-1}\}$, with a bounded $1$-potential, and a family of {\it excursion measures}, $(\mathbb{N}_{\theta }, \theta\in\mathbb{S}_{d-1})$, such that  
\begin{itemize}
\item[(i)] the map $\theta\mapsto \mathbb{N}_{\theta}$ is a kernel from $\mathbb{S}_{d-1}$ to $\mathbb{R}\times\mathbb{S}_{d-1},$ such that $\mathbb{N}_{\theta}(1-e^{-\zeta})<\infty$ and $\mathbb{N}_{\theta}$ is carried by the set $\{(\epsilon(0+),\Theta^\epsilon(0)=(0,\theta)\}$ and $\{\zeta>0\};$
\item[(ii)]we have the {\it exit formula} \begin{align}
&\mathbf{E}_{x,\theta}\left[\sum_{g\in G}F((\xi_s, \Theta_s): s<g)H((\epsilon_{g}, \Theta^\epsilon_g))\right]\notag\\
&\hspace{2cm}=\mathbf{E}_{x,\theta}\left[\int_0^\infty F((\xi_s, \Theta_s): s< t)\mathbb{N}_{ \Theta_t}(H(\epsilon, \Theta^\epsilon)){\rm d}\Lambda_t\right],
\label{exitsystem1}
\end{align}
for $x\neq 0$, where $F$ is continuous on the space of c\`adl\`ag paths $\mathbb{D}(\mathbb{R}\times\mathbb{S}_{d-1})$ and $H$ is measurable on the space of c\`adl\`ag paths $\mathbb{U}(\mathbb{R}\times\mathbb{S}_{d-1});$
\item[(iii)] under any measure $\mathbb{N}_{\theta}$ the process $(\epsilon, \Theta^\epsilon)$ is Markovian with the same semigroup as $(\xi, \Theta)$ stopped at its first hitting time of $(-\infty,0]\times\mathbb{S}_{d-1}.$   
\end{itemize}
The couple $(\Lambda, \mathbb{N}^{\cdot})$ is called an exit system. Note that in Maisonneuve's original formulation, the pair $\Lambda$ and the kernel $\mathbb{N}$ is not unique, but once $\Lambda$ is chosen the measures $(\mathbb{N}_\theta, \theta\in \mathbb{S}_{d-1})$ are determined but for a $\Lambda$-neglectable set, i.e. a set $\mathcal{A}$ such that $\mathbf{E}_{x,\theta}(\int_{t\geq 0}1_{\{(\xi_s-\underline{\xi}_s,\Theta_s)\in\mathcal{A}\}}{\rm d}\Lambda_s)=0$. Since $\ell$ is an additive functional with a bounded $1$-potential, we will henceforth work with the exit system $(\ell, \mathbb{N}^{\cdot})$ corresponding to it. 



\medskip

The  importance   of \eqref{exitsystem1} can already be seen when we consider the distribution of $X_{{\texttt G}(\infty)}$. Indeed, we have  for bounded measurable $f$ on $\mathbb{R}^d$,
\begin{align}
\mathbb{E}_x[f(X_{{\texttt G}(\infty)})] &= 
\mathbf{E}_{\log|x|, \arg(x)}\left[\sum_{t\in G}f({\rm e}^{\xi_{t}} \Theta_t) \mathbf{1}{(\zeta_{t} =\infty)}\right]\notag\notag\\
&=\mathbf{E}_{\log|x|, \arg(x)}\left[\int_0^\infty f({\rm e}^{\xi_{t}} \Theta_t) \mathbb{N}_{\Theta_t}(\zeta=\infty)
\d \ell_t\right]\notag\\
&=\mathbf{E}_{\log|x|, \arg(x)}\left[\int_0^{\ell_\infty} f({\rm e}^{-H^-_t}\Theta^-_t) \mathbb{N}_{\Theta^-_t}(\zeta=\infty)\d t\right]\notag\\
&=\int_{|z|<|x|}U^-_x(\d z)f(z)\mathbb{N}_{\arg(z)}(\zeta = \infty),
\label{X_g}
\end{align}
where 
\[
U^-_x(\d z) :=\int_0^{\infty} \mathbf{P}_{\log |x|,\arg(x)}({\rm e}^{-H^-_t}\Theta^-_t\in \d z, \, t<\ell_\infty) \d t, \qquad |z|\leq|x|
\]
may be thought of as a potential. 
\begin{rem}\rm 
It is worth noting here that the definition of  $U^-_x$ is designed specifically to look at the expected occupation measure of the radial minima in cartesian coordinates, rather than in polar coordinates which would be another natural potential associated with $(H^-_t, \Theta^-_t)$, $t\geq 0$.
\end{rem}

On account of the fact that $X$ is transient, in the sense of \eqref{transient}, we know that $(H^-, \Theta^-)$  experiences killing at a rate that occurs, in principle, in a state-dependent manner, specifically $\mathbb{N}_\theta(\zeta = \infty)$, $\theta\in\mathbb{S}_{d-1}$. Isotropy allows us to conclude that all such rates take a common value and thanks to the arbitrary scaling of local time $\ell$, we can choose this common value to be unity. Said another way, $\ell_\infty$ is exponentially distributed with rate $1$.

In conclusion, we reach the identity 
\begin{equation}
\mathbb{E}_x[f(X_{{\texttt G}(\infty)})] = \int_{|z|<|x|}U^-_x(\d z)f(z)
\label{Ef(X_G)}
\end{equation}
or equivalently, the law of $X_{{\texttt G}(\infty)}$ under $\mathbb{P}_x$, $x\neq 0$, is nothing more than the measure  $U^-_x(\d z)$, $|z|\leq |x|$.  From this analysis, in combination with \eqref{BGR2}, we also get another handy identity which will soon be of use. For $r<|x|$, $\mathbb{P}_x(\tau^\oplus_r=\infty) =\mathbb{P}_x(|X_{{\texttt G}(\infty)}|>r)$ and hence, from Theorem~\ref{BGRtheorem} we have
 \begin{align}\mathbb{P}_x(\tau^\oplus_r=\infty)&=
\int_{r<|z|<|x|}U^-_x(\d z)\notag\\
&= \frac{\Gamma(d/2)}{\Gamma((d - \alpha)/2)\Gamma(\alpha/2)}\int_0^{(|x|^2/r^2)-1} (u+1)^{-d/2}u^{\alpha/2-1}\d u.
\label{noenter}
\end{align}

Another identity where we gain some insight into the quantity $U^-_x$ is the first passage result of Blumental-Getoor-Ray \cite{BGR} which was already stated in \eqref{BGR}. 
For example, the following identity emerges very quickly from \eqref{exitsystem1}. For bounded measurable functions $f,g$ on $\mathbb{R}^d$, 
\begin{align}
&\mathbb{E}_x[g(X_{{\texttt G}(\tau_1^\oplus)})f(X_{\tau^\oplus_1}); \tau^\oplus_1<\infty]\notag\\
& =  \int_{1<|z|<|x|}U^-_x(\d z)\int_{|y||z|<1}\mathbb{N}_{\arg(z)}({e}^{\epsilon(\zeta)}\Theta^\epsilon(\zeta)\in \d y; \zeta<\infty)g(z)f( |z|y).
\label{two-law}
\end{align}
With judicious computations in the spirit of those given above, one might expect to be able to extract an identity for $U^-_x$ in combination with \eqref{BGR}. For example, developing \eqref{two-law} we might write
\begin{align}
\mathbb{E}_x[f(|X_{\tau^\oplus_1}|); \tau^\oplus_1<\infty] &=  \int_{1<|z|<|x|}U^-_x(\d z)\int_{y>\log|z|}\mathbb{N}_{\arg(z)}(|\epsilon(\zeta)|\in \d y; \zeta<\infty)f(|z| {\rm e}^{-y})\notag\\
&=\int_{1<|z|<|x|}U^-_x(\d z)\int_{y>\log |z|}\nu( \d y)f(|z| {\rm e}^{-y})
\label{invertthis?}
\end{align}
for $|x|>1$ and bounded measurable $f$ on $\mathbb{R}^d$, where we have appealed to  isotropy to ensure that $\mathbb{N}_{\arg(z)}(|\epsilon(\zeta)|\in \d y)$ does not depend on $\arg(z)$ and thus can rather be written as $\nu( \d y)$, where $\nu$ is therefore the L\'evy measure of the subordinator $H^-$, {\color{black}see e.g. \cite{V}.}
On account of the fact that the Wiener--Hopf factorisation for $\xi$ is known, c.f. \eqref{a}, the measure $\nu$ can written explicitly; see \cite{CPP11}. Indeed, the normalisation of $\ell$ 
is equivalent to the requirement that $\Phi^-(0)=1$, where $\Phi^-$ is the Laplace exponent of $H^-$ and hence 
\[
\Phi^-(\lambda) = \int_{(0,\infty)} (1- {\rm e}^{-\lambda y})\nu(\d y) = \dfrac{\Gamma((d-\alpha)/2)\Gamma((\lambda +d)/2)}{\Gamma({d}/{2})\Gamma((\lambda +d-\alpha)/2)},\qquad \lambda\geq 0,
\]
which, inverting with the help of a change of variables and the beta integral (see also \cite{CPP11}),  tells us that 
\begin{equation}
\nu(\d y) = \frac{\alpha\Gamma((d-\alpha)/2)}{\Gamma({d}/{2})\Gamma(1-{\alpha}/{2})}
(1-{\rm e}^{-2y})^{-\frac{\alpha}{2}-1}{\rm e}^{-d y}\d y.
\label{nu}
\end{equation}
Nonetheless, despite the fact that the left-hand side of \eqref{invertthis?} and \eqref{nu} are explicitly available, it seems here, and in other similar computations of this type, difficult to back out an expression for the measure  $U^-_x$.

Whilst our approach will make use of some of the identities above, fundamentally we prove Theorem \ref{main1} via a method of approximation, out of which the expression we will obtain for $U^-_x$ can be {\color{black} cleverly used, in conjunction of the excursion theory above,  to derive a number of other identities.} 

\section{Proof of Theorem \ref{main1}}\label{proofofmain1}
We start with some notation. First define, for $x\neq 0$, $|x|>r$,  $\delta>0$ and  continuous, positive and bounded $f$ on $\mathbb{R}^d$,
\[
 \Delta_r^{\delta}f(x):=\frac{1}{\delta}\mathbb{E}_x \left[f(\arg(X_{\texttt{G}_\infty})) ,|X_{\texttt{G}_\infty}|\in [r-\delta,r]\right].
 \]
The crux of our proof is to establish a limit of $ \Delta_r^{\delta}f(x)$ in concrete terms as $\delta\to0$.

Note that, by conditioning on first entry into the ball of radius $r$, we have, with the help of the first entrance law \eqref{BGR} and \eqref{Ef(X_G)},
\begin{align} 
 \Delta_r^{\delta}f(x)
&=\frac{1}{\delta}\int_{|y|\in[r-\delta, r]}\mathbb{P}_x(X_{\tau^\oplus_r}\in {\rm d} y;\, \tau^\oplus_r<\infty) \mathbb{E}_y\left[f(\arg(X_{\texttt{G}_\infty})); \, |X_{\texttt{G}_\infty}|\in   (r-\delta,|y|]\right]\notag\\
&=\frac{1}{\delta}C_{\alpha,d}\int_{|y|\in[r-\delta, r]} {\rm d} y \left| \frac{r^2-|x|^2}{r^2-|y|^2}  \right|^{\alpha/2}|y-x|^{-d}\mathbb{E}_y\left[f(\arg(X_{\texttt{G}_\infty}));\, |X_{\texttt{G}_\infty}|\in   (r-\delta,|y|]\right]\notag\\
&=\frac{1}{\delta}C_{\alpha,d}|r^2-|x|^2|^{\alpha/2}\int_{|y|\in(r-\delta, r]} {\rm d} y\frac{ |y-x|^{-d}}{|r^2-|y|^2|^{\alpha/2}} 
\int_{r-\delta\leq |z|\leq |y|} U^-_y(\d z)f(\arg(z)),
\label{Delta}\end{align}
where 
\[
C_{\alpha,d} = \pi^{-(d/2+1)}\Gamma\left({d}/{2} \right)\sin\left(\frac{\pi\alpha}{2}\right).
\]
Our next objective is to try and replace $\int_{r-\delta\leq |z|\leq |y|} U^-_y(\d z)f(\arg(z))$ by a term of simpler form which can be asymptotically estimated in the limit as $\delta\to0$. To this end, we need some technical lemmas.

\begin{lemma}\label{Lemma1}
Suppose that $f$ is a bounded continuous function on $\mathbb{R}^d$. Then
\[
\lim_{\delta\to0}\sup_{|y|\in(r-\delta,r]}\left|\frac{\int_{r-\delta\leq |z|\leq |y|} U^-_y(\d z)f(z)}{\int_{r-\delta\leq |z|\leq |y|} U^-_y(\d z)}- f(y)
\right|=0.
\]
\end{lemma}

\begin{proof}Suppose that  $\mathcal{C}_{r, \delta,\varepsilon}(y)$ is the geometric region which coincides with the intersection of a cone with axis along $y$ with radial extent $2\varepsilon$, say $\mathcal{C}_\varepsilon$, and the annulus $\{z\in\mathbb{R}^d: r-\delta\leq |z|\leq r\}$; see Figure \ref{conediag}. 
Chose $\varepsilon,\delta$ such that 
\[
\sup_{z\in\mathcal{C}_{r, \delta,\varepsilon}(y)}|f(z)-f(y)|<\varepsilon',
\]
for some choice of $\varepsilon'\ll 1$. 

\begin{figure}[h!]
\begin{center}
\includegraphics[width = 10cm]{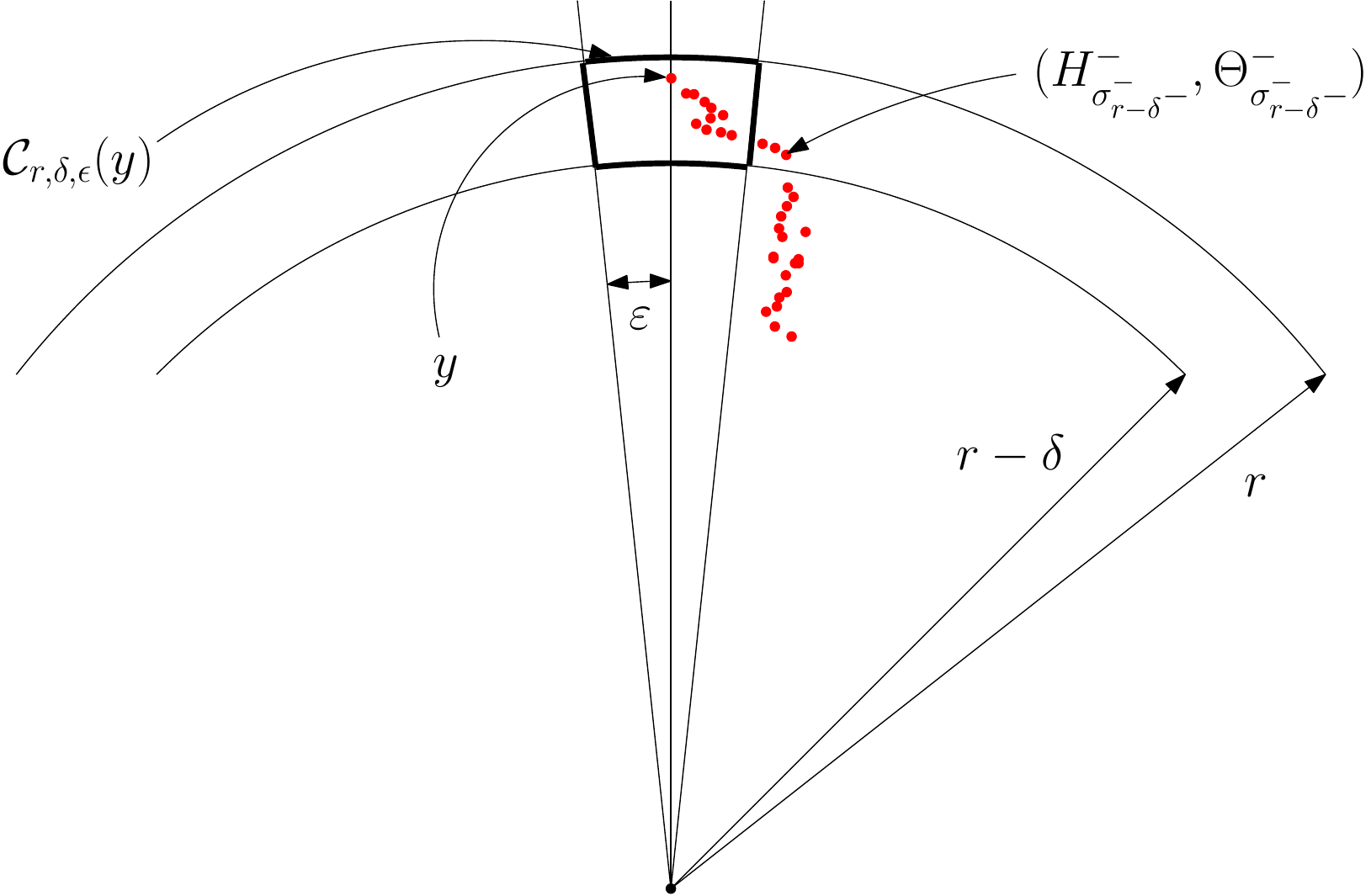}
\end{center}
\caption{The process $(H^-,\Theta^-)$ in relation to the domain $\mathcal{C}_{r, \delta,\varepsilon}(y)$.}
\label{conediag}
\end{figure}

We have
\begin{align}
&\sup_{|y|\in(r-\delta,r]}\left|\frac{\int_{r-\delta\leq |z|\leq |y|} U^-_y(\d z)f(z)}{\int_{r-\delta\leq |z|\leq |y|} U^-_y(\d z)}- f(y)
\right|\notag\\
&
\leq \varepsilon'+||f||_\infty
\sup_{|y|\in(r-\delta,r]}\frac{\int_{r-\delta\leq |z|\leq |y|} U^-_y(\d z)\mathbf{1}{(z\not\in \mathcal{C}_{r, \delta,\varepsilon}(y))}}{\int_{r-\delta\leq |z|\leq |y|} U^-_y(\d z)}.
\label{2ndterm}
\end{align}
In order to deal with the second term in the right-hand side above, taking the example computations of \eqref{two-law} and \eqref{invertthis?}, note that, for $|y|\in(r-\delta,r]$,
\begin{align}
\sup_{|y|\in(r-\delta,r]} 
\int_{r-\delta\leq |z|\leq |y|} & U^-_y(\d z)\mathbf{1}{(z\not\in\mathcal{C}_{r, \delta,\varepsilon}(y))} \nu\left( \log\left(\frac{|z|}{r-\delta}\right), \infty\right)\notag\\
&= \sup_{|y|\in(r-\delta,r]} 
\mathbb{P}_{y}(X_{\tau^\oplus_{r-\delta} - }\not\in\mathcal{C}_{r, \delta,\varepsilon}(y),\, \tau^\oplus_{r-\delta} <\infty)\notag\\
&= \sup_{\beta\in(r-\delta,r]} 
\mathbb{P}_{\beta\texttt{1}}(X_{\tau^\oplus_{r-\delta} - }\not\in\mathcal{C}_{r, \delta,\varepsilon}(\beta\texttt{1}),\,\tau^\oplus_{r-\delta} <\infty)\notag\\
&\leq
 \sup_{\beta\in(r-\delta,r]} 
 \mathbb{P}_{\beta\texttt{1}}(\Theta^-_{\sigma^-_{r-\delta} -}\not\in\mathcal{C}_\varepsilon\cap\mathbb{S}_{d-1},\,\sigma^-_{r-\delta}<\infty
)\notag\\
&\leq 
\sup_{\beta\in(r-\delta,r]}   \mathbb{P}_{\beta\texttt{1}}(\upsilon_\varepsilon <\sigma^-_{r-\delta}
)
\notag\\
&\leq\mathbb{P}_{r\texttt{1}}(\upsilon_\varepsilon <\sigma^-_{r-\delta}
)
\end{align}
where ${\texttt 1} = (1,0,\cdots,0)$ is the `North Pole' on $\mathbb{S}_{d-1}$, 
$
\sigma^-_{r-\delta} = \inf\{t>0: H^-_t<r-\delta\}
$
and $\upsilon_\varepsilon = \inf\{t>0: \Theta^-_t \not\in \mathcal{C}_\varepsilon\cap\mathbb{S}_{d-1}\}$.
Right-continuity of  paths now ensures that the right-hand side above tends to zero as $\delta\to0$.

On the other hand, from \eqref{noenter} 
\begin{align}
\label{itsaprobability}\int_{r-\delta\leq |z|\leq |y|} & U^-_y(\d z)
=\mathbb{P}_y(\tau^\oplus_{r-\delta}=\infty)  = \mathbb{P}_{\frac{|y|}{(r-\delta)}\texttt{1}}(\tau^\oplus_{1}=\infty),
 \end{align}
 where we have used isotropy in the final equality 
and from \eqref{BGR2} and  \eqref{nu} a rather elementary computation shows that
\begin{align*}
&\lim_{\eta\downarrow1}\nu\left( \log\eta, \infty\right)\mathbb{P}_{\eta\texttt{1}}(\tau^\oplus_{1}=\infty)\\
&= \lim_{\eta\downarrow1}
\frac{\alpha}{\Gamma({\alpha}/{2})\Gamma(1-{\alpha}/{2})}
\left(\int_{\log \eta}^\infty(1-{\rm e}^{-2v})^{-\frac{\alpha}{2}-1}{\rm e}^{-d v}\d v\right)
\left(\int_0^{\eta^2-1} (u+1)^{-d/2}u^{\alpha/2-1}\d u \right)\\
&=\frac{1}{\Gamma(1+{\alpha}/{2})\Gamma(1-{\alpha}/{2})}
\end{align*}
Hence 
\begin{align*}
&\lim_{\delta\to0}\sup_{|y|\in(r-\delta,r]}
\frac{\int_{r-\delta\leq |z|\leq |y|} U^-_y(\d z)\mathbf{1}{(z\not\in \mathcal{C}_{r, \delta,\varepsilon}(y))}}{\int_{r-\delta\leq |z|\leq |y|} U^-_y(\d z)}\\
&\leq\lim_{\delta\to0} \sup_{|y|\in(r-\delta,r]}
\frac{\int_{r-\delta\leq |z|\leq |y|} U^-_y(\d z)\mathbf{1}{(z\not\in \mathcal{C}_{r, \delta,\varepsilon}(y))}
\dfrac{\nu( \log\left({|z|}/{(r-\delta)}\right), \infty)}{\nu( \log\left({|z|}/{(r-\delta)}\right), \infty)}}
{\int_{r-\delta\leq |z|\leq |y|} U^-_y(\d z)}\\
&\leq\lim_{\delta\to0} \sup_{|y|\in(r-\delta,r]}
\frac{\int_{r-\delta\leq |z|\leq |y|} U^-_y(\d z)\mathbf{1}{(z\not\in \mathcal{C}_{r, \delta,\varepsilon}(y))}
\nu( \log\left({|z|}/{(r-\delta)}\right), \infty)}
{\nu( \log\left({|y|}/{(r-\delta)}\right), \infty) \mathbb{P}_y(\tau^\oplus_{r-\delta}=\infty)}\\
&\leq\lim_{\delta\to0}
\sup_{1<\eta<1+\frac{\delta}{(r-\delta)}}\frac{\mathbb{P}_{r\texttt{1}}(\upsilon_\varepsilon <\sigma^-_{r-\delta}
)
}{\nu( \log\eta, \infty) \mathbb{P}_{\eta\texttt{1}}(\tau^\oplus_{1}=\infty)}\\
&=0
\end{align*}
and thus plugging this back into \eqref{2ndterm} gives the result.
\end{proof}

With Lemma \ref{Lemma1} in hand, noting in particular the representation \eqref{itsaprobability}, we can now return to \eqref{Delta}  and note that, for each $\varepsilon>0$, we can choose $\delta$ sufficiently small such that
\[
 \Delta_r^{\delta}f(x) = D(\varepsilon) \Delta_r^{\delta}1(x) + \frac{1}{\delta}C_{\alpha,d}|r^2-|x|^2|^{\alpha/2}\int_{|y|\in(r-\delta, r]} {\rm d} y\frac{ |y-x|^{-d}}{|r^2-|y|^2|^{\alpha/2}}f(\arg(y)) \mathbb{P}_y(\tau^\oplus_{r-\delta}=\infty) ,
\]
where, $|D(\varepsilon)|<\varepsilon$ and for $|x|>r$,
\begin{align*}
\limsup_{\delta\to0}|\Delta_r^{\delta}1(x)|&\leq \limsup_{\delta\to0}\left|\frac{1}{\delta}C_{\alpha,d}|r^2-|x|^2|^{\alpha/2}\int_{|y|\in(r-\delta, r]} {\rm d} y\frac{ |y-x|^{-d}}{|r^2-|y|^2|^{\alpha/2}} \mathbb{P}_y(\tau^\oplus_{r-\delta}=\infty)\right|\\
&=\limsup_{\delta\to0}\left|\frac{1}{\delta}\left(\mathbb{P}_x(\tau^\oplus_{r-\delta}=\infty) - \mathbb{P}_x(\tau^\oplus_{r}=\infty)\right)\right|\\
&=\frac{\Gamma(d/2)}{\Gamma((d - \alpha)/2)\Gamma(\alpha/2)}
\left|\left.\frac{\d}{\d v}\int_0^{(|x|^2/v^2)-1} (u+1)^{-d/2}u^{\alpha/2-1}\d u\right|_{v=r}\right|\\
&=\frac{2\Gamma(d/2)}{\Gamma((d - \alpha)/2)\Gamma(\alpha/2)}
\left(|x|^2-r^2\right)^{\alpha/2 -1}r^{d-1-\alpha} |x|^{2-d}
\end{align*}
where in the third equality we have used \eqref{BGR2}. 

We can now say that, if the limit exists, 
\begin{align}
&\lim_{\delta\to0}\Delta_r^{\delta}f(x)\notag\\
&= \lim_{\delta\to0}C_{\alpha,d}|r^2-|x|^2|^{\alpha/2}
\frac{1}{\delta}
\int_{|y|\in(r-\delta, r]} {\rm d} y\frac{ |y-x|^{-d}}{|r^2-|y|^2|^{\alpha/2}}f(\arg(y)) \mathbb{P}_y(\tau^\oplus_{r-\delta}=\infty)\notag\\
&=\lim_{\delta\to0}C_{\alpha,d}|r^2-|x|^2|^{\alpha/2}
\frac{1}{\delta}\int_{r-\delta}^r  \rho^{d-1}\d \rho \int_{\rho\mathbb{S}_{d-1}} \sigma_\rho(\d\theta)
\frac{ |\rho\theta-x|^{-d}}{|r^2-\rho^2|^{\alpha/2}}f(\theta) \mathbb{P}_{\rho\theta}(\tau^\oplus_{r-\delta}=\infty)\notag\\
&=\lim_{\delta\to0}C_{\alpha,d}|r^2-|x|^2|^{\alpha/2}
\frac{1}{\delta}\int_{r-\delta}^r  \rho^{d-1} \d \rho
\frac{\mathbb{P}_{\rho\texttt{1}}(\tau^\oplus_{r-\delta}=\infty)}{|r^2-\rho^2|^{\alpha/2}}
\int_{\rho\mathbb{S}_{d-1}} \sigma_\rho(\d\theta)
|\rho\theta-x|^{-d}f(\theta) ,
\label{ifexists}
\end{align}
where, in the second equality, we have switched from $d$-dimensional Lebesgue measure to the generalised  polar coordinate measure $\rho^{d-1}\d \rho\times \sigma_\rho(\d\theta)$, so that  $\rho>0$ is the radial distance from the origin and $\sigma_\rho(\d\theta)$ is the surface measure on $\rho\mathbb{S}_{d-1}$, normalised to have unit mass. In the third equality we have used isotropy to write $\mathbb{P}_{\rho\theta}(\tau^\oplus_{r-\delta}=\infty) = \mathbb{P}_{\rho\texttt{1}}(\tau^\oplus_{r-\delta}=\infty)$ for $\theta\in\mathbb{S}_{d-1}$.

Noting the continuity of the integral $\int_{\rho\mathbb{S}_{d-1}} \sigma_\rho(\d\theta)
|\rho{\texttt 1}-x|^{-d}f(\theta) $ in $\rho$, the proof of Theorem \ref{main1} is complete as soon as we can evaluate 
\begin{equation}
\lim_{\delta\to0}\frac{1}{\delta}\int_{r-\delta}^r \rho^{d-1} \d \rho 
\frac{\mathbb{P}_{\rho\texttt{1}}(\tau^\oplus_{r-\delta}=\infty)}{|r^2-\rho^2|^{-\alpha/2}}.
\label{completedassoonas}
\end{equation}
To this end, we need a technical lemma.

\begin{lemma}\label{lemma:escape} Let  $
D_{\alpha,d} = {\Gamma(d/2)}/{\Gamma((d - \alpha)/2)\Gamma(\alpha/2)}
$. Then 
\[
\lim_{\delta\to0}\sup_{\rho\in[r-\delta,r]} \left|(\rho^2-(r-\delta)^2)^{-\alpha/2}r^{\alpha}
 \mathbb{P}_{
 \rho{\emph {\texttt 1}   }
 }
 (\tau^\oplus_{r-\delta}=\infty) - \frac{2D_{\alpha,d}}{\alpha}\right|=0 
\]
\end{lemma}
\begin{proof}Appealing to \eqref{BGR2}, we start by noting that 
\begin{align} &\sup_{\rho\in[r-\delta,r]}\left|D_{\alpha,d}\int_0^{\rho^2/(r-\delta)^2 -1}u^{\alpha/2-1}{\rm d}u - \mathbb{P}_{\rho{\texttt 1}}(\tau^\oplus_{r-\delta}=\infty)  \right| \nonumber\\
 &\hspace{2cm}\leq \sup_{\rho\in[r-\delta,r]}D_{\alpha,d}\int_0^{\rho^2/(r-\delta)^2 -1}\left|(1+ u)^{-d/2}-1\right|u^{\alpha/2-1}{\rm d}u  \nonumber\\
&\hspace{2cm}\leq \sup_{\rho\in[r-\delta,r]}D_{\alpha,d}\int_0^{\rho^2/(r-\delta)^2 -1}\left|1-\frac{(r-\delta)^d}{\rho^d}\right|u^{\alpha/2-1}{\rm d} u\notag\\
&\hspace{2cm}\leq D_{\alpha,d}\left|1-\frac{(r-\delta)^d}{r^d}\right|
\frac{2}{\alpha}\left(r^2-(r-\delta)^2\right)^{\alpha/2}(r-\delta)^{-\alpha},
\label{eq:escape1}
\end{align}
which tends to zero as $\delta\to0$. 
Furthermore,
\begin{align}&\sup_{\rho\in[r-\delta,r]}\left|D_{\alpha,d}\int_0^{\rho^2/(r-\delta)^2 -1}u^{\alpha/2-1}{\rm d}u-\frac{2 D_{\alpha,d}}{\alpha}(\rho^2-(r-\delta))^{\alpha/2}r^{-\alpha}\right|\nonumber\\
&=\sup_{\rho\in[r-\delta,r]}\frac{2 D_{\alpha,d}}{\alpha}(\rho^2-(r-\delta)^2 )^{\alpha/2}\left|(r-\delta)^{-\alpha}-r^{-\alpha}\right|\nonumber\\
&\leq \frac{2 D_{\alpha,d}}{\alpha}(r^2-(r-\delta))^{\alpha/2}\left|(r-\delta)^{-\alpha}-r^{-\alpha}\right|, \label{eq:escape2}
\end{align}
which also tends to zero as $\delta\to0$.
Summing \eqref{eq:escape1} and \eqref{eq:escape2} in the context of the  triangle inequality and dividing by $r^{-\alpha}(r^2-(r-\delta))^{\alpha/2}$ we can also deduce that
\[
\lim_{\delta\to0}\sup_{\rho\in[r-\delta,r]}\left|(\rho^2-(r-\delta)^2)^{-\alpha/2}r^{\alpha}
 \mathbb{P}_{
 \rho{ \texttt 1} 
 }
 (\tau^\oplus_{r-\delta}=\infty)- \frac{2D_{\alpha,d}}{\alpha}\right|=0,
\]
and the lemma is proved. 
\end{proof}

We are now ready to prove \eqref{completedassoonas}, and identify its limit, thereby completing the proof of Theorem \ref{main1}.
Appealing to Lemma \ref{lemma:escape}, for all  $\varepsilon>0$, there exists a $\delta$ sufficiently small,
\begin{align}
&\left|\frac{1}{\delta}\int_{r-\delta}^{r}{\rm d} \rho \frac{\mathbb{P}_{\rho{\texttt 1}} (\tau^\oplus_{r-\delta}=\infty)}{(r^2-\rho^2)^{\alpha/2}}  -\frac{2D_{\alpha,d}r^{-\alpha}}{\alpha}\frac{1}{\delta}\int_{r-\delta}^{r} {\rm d} \rho \frac{(\rho^2-(r-\delta)^2)^{\alpha/2}}{(r^2-\rho^2)^{\alpha/2}} \right|
\nonumber\\
&\hspace{8cm}
<\frac{\varepsilon}{\delta}\int_{r-\delta}^{r} {\rm d} \rho \frac{(\rho^2-(r-\delta)^2)^{\alpha/2}}{(r^2-\rho^2)^{\alpha/2}}.\label{ineq:A}
\end{align}
Next note that  
\begin{align}
\lim_{\delta\to0}\frac{1}{\delta}\int_{r-\delta}^{r} {\rm d} \rho \frac{(\rho^2-(r-\delta)^2)^{\alpha/2}}{(r^2-\rho^2)^{\alpha/2}}\notag
&= \lim_{\delta\to 0}\frac{1}{\delta}\int_{r-\delta}^{r} {\rm d} \rho \left[ \frac{\rho-(r-\delta)}{r-\rho}  \right]^{\alpha/2}\left[ \frac{\rho+(r-\delta)}{r+\rho}  \right]^{\alpha/2}\notag\\
&=\lim_{\delta\to 0}\int_{0}^{1} {\rm d} u \left[ \frac{u}{1-u}  \right]^{\alpha/2}\left[ \frac{2r-2\delta+\delta u}{2r-\delta+\delta u}  \right]^{\alpha/2}\notag\\
&=\int_{0}^{1} {\rm d} u (1-u)^{-\alpha/2}u^{\alpha/2} \notag\\
&=\Gamma(1-\alpha/2)\Gamma(1+\alpha/2),
\label{doublegamma}
\end{align}
where we have used the substitution $\rho=(r-\delta)+u\delta$ in the second equality and dominated convergence in the third.

Putting the pieces together, we can take limits in \eqref{ineq:A}, using  \eqref{doublegamma}, to deduce that 
\[
\lim_{\delta\to0}\frac{1}{\delta}\int_{r-\delta}^{r}{\rm d} \rho \frac{\mathbb{P}_{\rho{\texttt 1}} (\tau^\oplus_{r-\delta}=\infty)}{(r^2-\rho^2)^{\alpha/2}}  =\frac{2}{\alpha}D_{\alpha,d} \Gamma(1-\alpha/2)\Gamma(1+\alpha/2)r^{-\alpha}
\]
which, in turn, can be plugged  into  \eqref{ifexists} and we find that 
\begin{align*}
\lim_{\delta\to0}\Delta^\delta_rf(x) &= \frac{2}{\alpha}D_{\alpha,d}\Gamma(1-\alpha/2)\Gamma(1+\alpha/2)C_{\alpha,d}
r^{d-\alpha-1}
|r^2-|x|^2|^{\alpha/2}
\int_{r\mathbb{S}_{d-1}} \sigma_\rho(\d\theta)
|r\theta-x|^{-d}f(\theta)\\
&= \pi^{-d/2}\frac{\Gamma(d/2)^2}{\Gamma((d - \alpha)/2)\Gamma(\alpha/2)}
r^{d-\alpha-1}
|r^2-|x|^2|^{\alpha/2}
\int_{r\mathbb{S}_{d-1}} \sigma_\rho(\d\theta)|r\theta-x|^{-d}f(\theta).
\end{align*}
Now suppose that $g$ is another bounded measurable function on $[0,\infty)$, then 
\begin{align*}
&\mathbb{E}_x[g(|X_{{\texttt G}(\infty)}|)f(\arg(X_{{\texttt G}(\infty)}))]\\
&= \pi^{-d/2}\frac{\Gamma(d/2)^2}{\Gamma((d - \alpha)/2)\Gamma(\alpha/2)}
\int_0^{|x|}\int_{r\mathbb{S}_{d-1}}  r^{d-1} \d r\sigma_\rho(\d\theta)
\frac{|r^2-|x|^2|^{\alpha/2}}
{r^\alpha|r\theta-x|^{d}}f(\theta)g(r)\\
&= \pi^{-d/2}\frac{\Gamma(d/2)^2}{\Gamma((d - \alpha)/2)\Gamma(\alpha/2)}\int_{|y|<|x|}\frac{||y|^2-|x|^2|^{\alpha/2}}
{|y|^\alpha|y-x|^{d}}f(\arg(y))g(|y|)\d y,
\end{align*}
which is equivalent to the statement of Theorem \ref{main1}.
 \hfill$\square$

\section{Riesz--Bogdan--\.Zak transform and MAP duality}\label{sectionRBZ}
Recently, Bogdan and \.Zak \cite{BZ}  used an idea of Riesz from classical potential analysis to 
understand the relationship between a stable process and its transformation through a simple sphere inversion. (See also Alili et al. \cite{ACGZ} and Kyprianou \cite{Deep1}).
Suppose we write $Kx = x/|x|^2$, $x\in\mathbb{R}^d$ for the classical inversion of space through the sphere $\mathbb{S}_{d-1}$. Then in dimension $d\geq 2$,  Bogdan and \.Zak \cite{BZ} prove that, for $x\neq 0$, $(KX_{\eta(t)}, t\geq 0)$ under $
\mathbb{P}_{Kx}$ is equal in law to $(X_t, t\geq 0)$ under $\mathbb{P}^\circ_{x}$, where
\begin{equation}
\left.\frac{\d \mathbb{P}^\circ_x}{\d \mathbb{P}_x}\right|_{\sigma(X_s: s\leq t)} = \frac{|X_t|^{\alpha-d}}{|x|^{\alpha-d}}, \qquad t\geq 0
\label{COM}
\end{equation}
and $\eta(t) = \inf\{s>0: \int_0^s|X_u|^{-2\alpha}\d u>t\}$. 
It was shown in Kyprianou et al. \cite{KRS} that $\mathbb{P}^\circ_x$, $x\in\mathbb{R}^d\backslash\{0\}$ can be understood, in the appropriate sense, as the stable process conditioned to be continuously absorbed at the origin. Indeed, as far as the underlying MAP $(\xi, \Theta)$ is concerned, we see that $-{\rm i}(\alpha- d)$ is a root of the exponent \eqref{a} and the change of measure \eqref{COM} corresponds to an Esscher transform of the L\'evy process $\xi$, rendering it a process which drifts to $-\infty$. Thus, an application of the optimal stopping theorem shows that  \eqref{COM} is equivalent to the change of measure for $\xi$
\begin{equation}
\left.\frac{\d \mathbf{P}^\circ_{x, \theta}}{\d \mathbf{P}_{x,\theta}}\right|_{\sigma((\xi_s, \Theta_s): s\leq t)} = {\rm e}^{(\alpha-d)(\xi_t-x)}, \qquad t\geq 0
\label{COMMAP}
\end{equation}

Following the reasoning in the one-dimensional case in \cite{ACGZ, Deep1}, it is not difficult to show that  the space-time transformed process $(KX_{\eta(t)}, t\geq 0)$ is the Lamperti--Kiu transform of the MAP $(-\xi, \Theta)$. Therefore, at the level of MAPs, the Riesz--Bogdan--\.Zak transform says that $(\xi, \Theta)$ under the change of measure \eqref{COMMAP}, when issued from $(\log|x|, \arg(x))$, $x\in\mathbb{R}$, is equal in law to $(-\xi, \Theta)$ when issued from $(-\log|x|, \arg(x))$.

An interesting consequence of this is that the Riesz--Bogdan--\.Zak transform provides an efficient way to analyse radial ascending properties of $X$, where previously we have studied its descending properties. That is to say, it offers the opportunity to study aspects of  the process $(H^+, \Theta^+)$. A good case in point in this respect is the analogue of the potential $U^-_x(\d y)$, $|y|<|x|$.

For convenience, note from Theorem \ref{main1} and \eqref{Ef(X_G)} that establishing the law of $X_{{\texttt G}(\infty)}$ is equivalent to obtaining an explicit identity for $U^-_x(\d y)$, $|y|<|x|$ and this we have already done. Specifically, for all $|x|>0$,
\begin{equation}
U^-_x(\d y) =  \pi^{-d/2}\frac{\Gamma(d/2)^2}{\Gamma((d - \alpha)/2)\Gamma(\alpha/2)}\frac{||y|^2-|x|^2|^{\alpha/2}}
{|y|^\alpha|y-x|^{d}}, \qquad |y|<|x|.
\label{U-}
\end{equation}
On the other hand, recalling that $\lim_{t\to\infty}|X_t| = \infty$, which implies that $\lim_{t\to\infty}\xi_t = \infty$ and hence $L_\infty = \infty$, we define
\[
U^+_x(\d z) =\int_0^{\infty} \mathbf{P}_{\log |x|,\arg(x)}({\rm e}^{H^+_t}\Theta^+_t\in \d z) \d t, \qquad |z|\geq|x|.
\]
Then the Riesz--Bogdan--\.Zak transform ensures that, for Borel $A\subseteq\{z\in\mathbb{R}^d: |z|<|x|\}$, 
\[
\frac{|z|^{\alpha-d}}{|x|^{\alpha-d}}\mathbf{P}_{\log |x|,\arg(x)}({\rm e}^{H^+_t}\Theta^+_t\in A)=\mathbf{P}_{-\log |x|,\arg(x)}({\rm e}^{-H^-_t}\Theta^-_t\in KA, t< \ell_{\infty})
\]
where $KA = \{Kz: z\in A\}$.
Hence, for $|x|>0$,
\begin{align}
U^+_x(\d z)& = \pi^{-d/2}\frac{\Gamma(d/2)^2}{\Gamma((d - \alpha)/2)\Gamma(\alpha/2)}\frac{||z|^{-2}-|x|^{-2}|^{\alpha/2}}
{|z|^{-\alpha}\left|({z}/{|z|^2})-({x}/{|x|^2})\right|^{d}}
\frac{|x|^{\alpha-d}}{|z|^{\alpha-d}}\frac{\d z}{|z|^{2d}}\notag\\
&= \pi^{-d/2}\frac{\Gamma(d/2)^2}{\Gamma((d - \alpha)/2)\Gamma(\alpha/2)}
\frac{||z|^2-|x|^2|^{\alpha/2}}
{|z|^\alpha|x-z|^{d} }, \qquad |z|>|x|.
\label{dualpotential}
\end{align}
where we have used the fact that $\d y = |z|^{-2d}\d z$, when $y = Kz$, and 
\begin{equation}
|Kx - Kz| = \frac{|x-z|}{|x||z|}.
\label{Kz-Kz}
\end{equation}

One notices that the identities for  the potential measures $U^-_x(\d z)$ and $U^+_x(\d z)$ are identical albeit that the former is supported on $|z|<|x|$ and the latter on $|z|>|x|$. These identities and, more generally, the duality that emerges through the Riesz--Bogdan--\.Zak transformation will be of use to us in due course. 

\section{Integration with respect to the excursion measure}\label{IntegrationN}

In order to proceed with some of the other fluctuation identities and the deep factorisation, we need to devote some time to compute in explicit detail the excursion occupation functionals
\begin{equation}
\mathbb{N}_\theta\left(\int_0^\zeta g({\rm e}^\epsilon(s)\Theta^\epsilon(s))\d s\right),\qquad \theta\in\mathbb{S}_{d-1},
\label{N1}
\end{equation}
and the excursion overshoot
\begin{equation}
\mathbb{N}_\theta\left(f({\rm e}^{\epsilon(\zeta)}\Theta^\epsilon(\zeta) );\,\zeta<\infty\right),\qquad \theta\in\mathbb{S}_{d-1},
\label{N2}
\end{equation}
for judicious choices of $f$ and $g$ that ensure  these quantities are finite.

The way we do this is to use Lemma \ref{Lemma1} to scale out the quantity of interest from a fluctuation identity in which it is placed together with the potential $U^-_x$. Let us start with the excursion overshoot in \eqref{N2}. 

\begin{proposition}\label{Nzeta}for $\theta\in\mathbb{S}_{d-1}$, we have 
\begin{align*}
&\mathbb{N}_{\theta}\left({\rm e}^{\epsilon(\zeta)}\Theta^\epsilon(\zeta) \in \d y; \zeta<\infty\right) \\
&\hspace{2cm}= 
\frac{\alpha\pi^{-d/2}}{2}
 \frac{ \Gamma((d - \alpha)/2)}{\Gamma(1-\alpha/2)}
|1-|y|^2|^{-\alpha/2}|\theta-y|^{-d}\d y, \qquad |y|\leq 1.
\end{align*}
\end{proposition}

\begin{proof}Take $|x|>r>r_0>0$ and suppose that $f: \mathbb{R}^d\mapsto[0,\infty)$ is continuous with support which is compactly embedded in the  ball of radius $r_0$.
 We have, on the one hand, from \eqref{BGR}, the identity
\[
\mathbb{E}_x[f(X_{\tau_r^\oplus}); \tau^\oplus_r<\infty]=\pi^{-(d/2+1)}\Gamma\left({d}/{2} \right)\sin\left(\frac{\pi\alpha}{2}\right) 
\int_{|y|<r} \frac{|r^2-|x|^2|^{\alpha/2}}{|r^2-|y|^2|^{\alpha/2}}|x-y|^{-d}f(y){\rm d} y.
\]
On the other hand, from \eqref{two-law}, we also have 
\begin{equation}
\mathbb{E}_x[f(X_{\tau^\oplus_r}); \tau^\oplus_r<\infty]\notag\\
=\int_{r<|z|<|x|}U^-_x(\d z)\int_{|y||z|<r}\mathbb{N}_{\arg(z)}(f(|z|{e}^{\epsilon(\zeta)}\Theta^\epsilon(\zeta)); \zeta<\infty).
\label{subin}
\end{equation}
Note that, for each $z\in\mathbb{R}^d\backslash\{0\}$, 
\[
z\mapsto\mathbb{N}_{\arg(z)}(f(|z|{e}^{\epsilon(\zeta)}\Theta^\epsilon(\zeta)); \zeta<\infty)
\]
is bounded thanks to the fact that $f$ is bounded and its support is compactly embedded in the unit ball or radius $r_0$. Indeed, there exists an $\varepsilon>0$, which depends only on the support of $f$, such that
\[
\sup_{r<|z|<|x|}\left|\mathbb{N}_{\arg(z)}(f(|z|{e}^{\epsilon(\zeta)}\Theta^\epsilon(\zeta)); \zeta<\infty)\right|\leq ||f||_\infty \nu(-\log (r_0-\varepsilon),\infty)<\infty.
\]

Moreover, since we can write 
\begin{equation}
\label{cts}\mathbb{N}_{\arg(z)}(f(|z|{e}^{\epsilon(\zeta)}\Theta^\epsilon(\zeta)); \zeta<\infty)
=\mathbb{N}_{\texttt{1}}(f(|z|{e}^{\epsilon(\zeta)}\Theta^\epsilon(\zeta)\star\arg(z)); \zeta<\infty),
\end{equation}
where, for any $a\in\mathbb{S}_{d-1}$,   the operation $\star \, a$  rotates the sphere so that the `North Pole', $\texttt{1} = (1,0,\cdots,0)\in\mathbb{S}_{d-1}$ moves to $a$. Using a straightforward dominated convergence argument, we see that $\mathbb{N}_{\arg(z)}(f(|z|{e}^{\epsilon(\zeta)}\Theta^\epsilon(\zeta)); \zeta<\infty)$ is continuous in $z$ thanks to the continuity of $f$.

Appealing to Lemma \ref{Lemma1}, we thus have that 
\begin{align*}
&\mathbb{N}_{\arg(x)}(f(|x|{e}^{\epsilon(\zeta)}\Theta^\epsilon(\zeta)); \zeta<\infty)\\
&\hspace{2cm}=\lim_{r\uparrow|x|}\frac{\int_{r<|z|<|x|}U^-_x(\d z)\int_{|y||z|<r}\mathbb{N}_{\arg(z)}(f(|z|{e}^{\epsilon(\zeta)}\Theta^\epsilon(\zeta)); \zeta<\infty)}{\int_{r<|z|\leq |x|}U^-_x(\d z)}\\
&\hspace{2cm}=\lim_{r\uparrow|x|}\frac{\mathbb{E}_x[f(X_{\tau^\oplus_r}); \tau^\oplus_r<\infty]}{\mathbb{P}_x(\tau^\oplus_r = \infty)}.
\end{align*}
Substituting in the analytical form of the ratio on the right-hand side above using \eqref{subin} and \eqref{BGR2}, we may continue with 
\begin{align}
&\mathbb{N}_{\arg(x)}(f(|x|{e}^{\epsilon(\zeta)}\Theta^\epsilon(\zeta)); \zeta<\infty)\notag\\
&= \lim_{r\uparrow|x|}\pi^{-d/2}
 \frac{ \Gamma((d - \alpha)/2)}{\Gamma(1-\alpha/2)}
\frac{ 
(|x|^2-r^2)^{\alpha/2}\int_{|y|<r} |r^2-|y|^2|^{-\alpha/2}|x-y|^{-d}f(y){\rm d} y}{ 
\int_0^{(|x|^2-r^2)/r^2} (u+1)^{-d/2}u^{\alpha/2-1}\d u}\notag\\
&=\pi^{-d/2}
 \frac{ \Gamma((d - \alpha)/2)}{\Gamma(1-\alpha/2)}
\int_{|y|<|x|} ||x|^2-|y|^2|^{-\alpha/2}|x-y|^{-d}f(y){\rm d} y\notag\\
&\hspace{7cm}\times\lim_{r\uparrow|x|}\frac{r^\alpha[(|x|^2-r^2)/r^2]^{\alpha/2}}{
\int_0^{(|x|^2-r^2)/r^2} (u+1)^{-d/2}u^{\alpha/2-1}\d u}\notag\\
&=\frac{\alpha\pi^{-d/2}}{2}
 \frac{ \Gamma((d - \alpha)/2)}{\Gamma(1-\alpha/2)}
\int_{|y|<|x|}|x|^\alpha ||x|^2-|y|^2|^{-\alpha/2}|x-y|^{-d}f(y){\rm d} y,
\label{morespecifically}
\end{align}
where we have used that the support of $f$ is compactly embedded in the ball of radius $|x|$ to justify the first term in the second equality.
\end{proof}

\medskip

Next we turn our attention to the quantity \eqref{N1}. Once again, our approach will be to scale an appropriate fluctuation identity by $\mathbb{P}_x(\tau^\oplus_r = \infty) = \int_{r<|z|\leq |x|}U^-_x(\d z)$. In this case, the natural object to work with is the expected occupation measure until first entry into the ball of radius $r<|x|$, where $x$ is the point of issue of the stable process. That is, the quantity
\begin{equation}
\mathbb{E}_x\left[\int_0^{\tau^\oplus_r}f(X_s) \d s\right] 
\label{resolvent}
\end{equation}
for $|x|>r>0$ and continuous $f:\mathbb{R}^d\mapsto[0,\infty)$ with compact support. Although an identity for the aforesaid resolvent is not readily available in the literature, it is not difficult to derive it from \eqref{BGR3}, with the help of  the Riesz--Bogdan--\.Zak transform. Recall that this transform states that, for $x\neq 0$, $(KX_{\eta(t)}, t\geq 0)$ under $
\mathbb{P}_{Kx}$ is equal in law to $(X_t, t\geq 0)$ under $\mathbb{P}^\circ_{x}$, where
 $\eta(t) = \inf\{s>0: \int_0^s|X_u|^{-2\alpha}\d u>t\}$.

For convenience, set $r=1$. Noting that, since $\int_0^{\eta(t)}|X_u|^{-2\alpha}\d u=t$, if we write $s = \eta(t)$, then 
\[
|X_{s}|^{-2\alpha}\d s =\d t, \qquad t>0,
\]
and hence we have that, for $|x|>1$,
\begin{align*}
 \int_{|z|>1}\frac{|z|^{\alpha-d}}{|x|^{\alpha -d}}h^\oplus_r(x,z)f(z)\d z&= \mathbb{E}^\circ_{x}\left[\int_0^{\tau^{\oplus}_1}f(X_t)\d t\right]\\
&=\mathbb{E}_{Kx}\left[\int_0^{\tau^\ominus_1} f(KX_{\eta(t)})
\d t\right] \\
&=\mathbb{E}_{Kx}\left[\int_0^{\tau^\ominus_1} f(KX_{s}) |X_s|^{-2\alpha}
\d s\right] \\
&=\int_{|y|<1} h^\ominus_1(Kx,y)f(K y)|y|^{-2\alpha}\d y
\end{align*}
where we have pre-emptively assumed that the resolvent associated to \eqref{resolvent} has a density, which we have denoted by $h^\oplus_1 (x,y)$. In the integral on the left-hand side above, we can make the change of variables $y= Kz$, which is equivalent to $z= Ky$. Noting that $\d y = \d z/{|z|^{2d}}$ and appealing to \eqref{BGR3}, we get
\begin{align*}
 \int_{|y|>1}\frac{|z|^{\alpha-d}}{|x|^{\alpha -d}}h^\oplus_1(x,z)f(z)\d z&=
 \int_{|z|>1} h^\ominus_1(Kx,Kz)f(z)
  \frac{|z|^{2\alpha}}{|z|^{2d}}\d z,
 \end{align*}
from which we can conclude that, for $|x|, |z|>1$,
\begin{align*}
h^\oplus_1(x,z)& = \frac{|x|^{\alpha -d}}{|z|^{\alpha-d}}h^\ominus_1(Kx,K z) \frac{ |z|^{2\alpha}}{|z|^{2d}}\notag\\
&=2^{-\alpha}\pi^{-d/2}\frac{\Gamma(d/2)}{\Gamma(\alpha/2)^2}\frac{|x|^{\alpha -d}}{|z|^{\alpha-d}}
\frac{ |z|^{2\alpha}}{|z|^{2d}}
|Kx-Kz|^{\alpha - d} \int_0^{\zeta^\ominus_1(Kx,Kz)}  (u+1)^{-d/2}u^{\alpha/2-1}\d u.
\end{align*}
Hence, after a little algebra, for $|x|, |z|>1$,
\begin{equation*}
h^\oplus_1(x,z)=2^{-\alpha}\pi^{-d/2}\frac{\Gamma(d/2)}{\Gamma(\alpha/2)^2}
|x-z|^{\alpha - d}
 \int_0^{\zeta^\oplus(x,z)}  (u+1)^{-d/2}u^{\alpha/2-1}\d u
\end{equation*}
where we have again used the fact that 
$|Kx - Kz| = |x-z|/|x||z|$ so that 
\[
 \zeta^\ominus_1(Kx,Kz)=(|x|^2-1)(|z|^{2}-1)/|x-z|^2 =: \zeta^\oplus_1(x,z).
\]
After scaling this gives us a general formula for \eqref{resolvent}, which we record below as a lemma on account of the fact that it does not already appear elsewhere in the literature (albeit being implicitly derivable as we have done from \cite{BGR}). 
\begin{lemma}\label{resolventin} For $|x|>r$, the resolvent  \eqref{resolvent} has a density given by 
\begin{equation}
h^\oplus_r(x,z) =2^{-\alpha}\pi^{-d/2}\frac{\Gamma(d/2)}{\Gamma(\alpha/2)^2}
|x-z|^{\alpha - d}
 \int_0^{\zeta^\oplus_r(x,z)}  (u+1)^{-d/2}u^{\alpha/2-1}\d u,
 \label{minustoplus}
\end{equation}
where $ \zeta^\oplus_r(x,z):= (|x|^2-r^2)(|z|^2-r^2)/r^2|x-z|^2$.
\end{lemma}

We can now use the above lemma to compute occupation potential with respect to the excursion measure. As for other results in this development, the following result is reminiscent of a classical result in fluctuation theory of L\'evy processes, see e.g. exercise 5 in Chapter VI in \cite{bertoin}, but as it includes the information about the modulator there is no direct way to derive it from the classical result. 

\begin{proposition}\label{Npotential}
For $x\in\mathbb{R}^d\backslash\{0\}$, and continuous  $g:\mathbb{R}^d\mapsto\mathbb{R}$ whose support is compactly embedded in the exterior of the ball of radius $|x|$,
\begin{align*}
\mathbb{N}_{\arg(x)}\left(\int_0^{\zeta}g(|x|{\rm e}^{\epsilon(u)}\Theta^\epsilon(u)){\rm d} u\right)
&= 2^{-\alpha}
\frac{\Gamma((d - \alpha)/2)^2}{\Gamma(d/2)^2}\int_{|x|<|z|} g(z) U^+_x(\d z)
\end{align*}
\end{proposition}
\begin{proof} Fix $0<r<|x|$.
Recall from the Lamperti--Kiu representation \eqref{eq:lamperti_kiu} that $X_t = \exp\{\xi_{\varphi(t)}\}\Theta(\varphi(t))$, $t\geq 0$, where $\int_0^{\varphi(t)}\exp\{\alpha \xi_u\}\d u = t$.
In particular, this implies that, if we write $s  = \varphi(t)$, then 
\begin{equation}
{\rm e}^{\alpha\xi_s} {\d}s= {\d t}, \qquad t>0,
\label{changeofvariableXtoMAP}
\end{equation}
Splitting the occupation over individual excursions, we have with the help of \eqref{exitsystem1} that 
\begin{align}
&\mathbb{E}_x\left[\int_0^{\tau^\oplus_r}g(X_t)\d t \right] \notag \\
&\hspace{1cm}=\mathbb{E}_x\left[\int_0^{\infty}\mathbf{1}{({\rm e}^{\underline{\xi}_s} > r)}g({\rm e}^{\xi_s}\Theta_s){\rm e}^{\alpha \xi_s}\d s\right] \notag \\
&\hspace{1cm}= 
\int_{r<|z|<|x|}U^-_x(\d z)\mathbb{N}_{\arg(z)}\left(\int_0^\zeta g(|z|{\rm e}^{\epsilon(s)} \Theta^\epsilon(s))(|z|{\rm e}^{\epsilon(s)})^\alpha\d s\right).
\label{occupation}
\end{align}
Note that the left-hand side is necessarily finite as it can be upper bounded by $\mathbb{E}_x\left[\int_0^\infty g(X_t)\d t \right]$, which is known to be finite for the given assumptions on $g$.
Straightforward  arguments, similar to those presented around \eqref{cts}, tell us that for continuous $g$ with compact support that is compactly embeded  in the exterior of ball of radius $|x|$, we have that, for $r<|z|<|x|$, 
\[
\mathbb{N}_{\arg(z)}\left(\int_0^\zeta g(|z|{\rm e}^{\epsilon(s)} \Theta^\epsilon(s)){\rm e}^{\alpha \epsilon(s)}\d s\right)
=\int_0^\infty 
\mathbb{N}_{\arg(z)}\left( g(|z|{\rm e}^{\epsilon(s)} \Theta^\epsilon(s)){\rm e}^{\alpha \epsilon(s)}; \, s<\zeta \right)
\d s
\]
is a continuous function.  Accordingly we can again use Lemma \ref{Lemma1} and Theorem \ref{BGRtheorem} and write, for $x\in\mathbb{R}^d$, 
\begin{align*}
&
\mathbb{N}_{\arg(x)}\left(\int_0^\zeta g(|x|{\rm e}^{\epsilon(s)} \Theta^\epsilon(s))(|x|{\rm e}^{ \epsilon(s)})^\alpha\right)\\
&=\lim_{r\uparrow|x|}\frac{\int_{r<|z|<|x|}U^-_x(\d z)
\mathbb{N}_{\arg(z)}\left(\int_0^\zeta g(|z|{\rm e}^{\epsilon(s)} \Theta^\epsilon(s))(|z|{\rm e}^{ \epsilon(s)})^\alpha\right)}{\int_{r<|z|\leq |x|}U^-_x(\d z)}\\
&=\frac{\mathbb{E}_x\left[\int_0^{\tau^\oplus_r}g(X_s)\d s \right] }{\mathbb{P}_x(\tau^\oplus_r = \infty)}\\
&=2^{-\alpha}\pi^{-d/2}
\frac{\Gamma((d - \alpha)/2)}{\Gamma(\alpha/2)}
\lim_{r\uparrow|x|}\frac{\int_{|x|<|z|}\d z \, \mathbf{1}{(r<|z|)}g(z)
|x-z|^{\alpha - d}
 \int_0^{\zeta^\oplus_r(x,z)}  (u+1)^{-d/2}u^{\alpha/2-1}\d u}{\int_0^{(|x|^2-r^2)/r^2} (u+1)^{-d/2}u^{\alpha/2-1}\d u}\\
&=2^{-\alpha}\pi^{-d/2}\frac{\Gamma((d - \alpha)/2)}{\Gamma(\alpha/2)}
\int_{|x|<|z|}\d z \, g(z) |x-z|^{- d}(|z|^2-|x|^2)^{\alpha/2},
\end{align*}
where in the final equality we have used dominated convergence (in particular the assumption on the support of $g$).
By inspection, we also note that the right-hand side above is equal to 
\[
2^{-\alpha}
\frac{\Gamma((d - \alpha)/2)^2}{\Gamma(d/2)^2}\int_{|x|<|z|} g(z)|z|^\alpha U^+_x(\d z).
\]
The proof is completed by replacing $g(x)$ by $g(x)|x|^{-\alpha}$. 
\end{proof}

\section{On $n$-tuple laws}\label{triplelaws}
We are now ready to prove Theorems \ref{tripleintheorem} \ref{double1} and \ref{double2} with the help of Section \ref{IntegrationN} and other identities. In essence, we can piece together the desired results using Maisonneuve's exit formula \eqref{exitsystem1} applied in the appropriate way, together with some of the identities established in previous section. 

\begin{proof}[Proof of Theorem \ref{tripleintheorem}]  (i) Appealing to the fact that the stable process $|X|$ does not creep downward and the L\'evy system compensation formula for the jumps of $X$, we have, on the one hand,
\begin{align}
\mathbb{E}_x[f(X_{{\texttt G}(\tau_r^\oplus)})g(X_{\tau^\oplus_r-})h(X_{\tau^\oplus_r}); \tau^\oplus_r<\infty] &=
\mathbb{E}_x\left[\int_0^{\tau^\oplus_r}f(X_{{\texttt G}(t)})g(X_t)k(X_t){\rm d}t\right],
\label{Levysystem}
\end{align}
where continuous $R$-valued functions $f$, $g$, $h$ are such that the first two are compactly supported in $\{z\in\mathbb{R}^d: |z|>r\}$ and the third is compactly supported in the open ball of radius $r$ and
\[
k(y) = \int_{|y+w|<r}\Pi(\d w)h(y+w).
\]
On the other hand, a calculation similar in spirit to \eqref{occupation}, using \eqref{exitsystem1}, followed by an application of Proposition \ref{Npotential}, tells us that
\begin{align*}
&\mathbb{E}_x\left[\int_0^{\tau^\oplus_r}f(X_{{\texttt G}(t)})g(X_t)k(X_t)\right]\notag\\
&\hspace{1cm}=\int_{r<|z|<|x|}U^-_x(\d z)f(z)
\mathbb{N}_{\arg(z)}\left(\int_0^\zeta g(|z|{\rm e}^{\epsilon(s)} \Theta^\epsilon(s))k(|z|{\rm e}^{\epsilon(s)} \Theta^\epsilon(s))(|z|{\rm e}^{\epsilon(s)})^\alpha\d s\right)\\
&\hspace{1cm}=2^{-\alpha}
\frac{\Gamma((d - \alpha)/2)^2}{\Gamma(d/2)^2}\int_{r<|z|<|x|}U^-_x(\d z)f(z)\int_{|z|<|y|}  U^+_z(\d y)g(y)k(y) |y|^\alpha.
\end{align*}
Putting the pieces together, we get 
\begin{align*}
&\mathbb{E}_x[f(X_{{\texttt G}(\tau_r^\oplus)})g(X_{\tau^\oplus_r-})h(X_{\tau^\oplus_r}); \tau^\oplus_r<\infty] \\
&=2^{-\alpha}
\frac{\Gamma((d - \alpha)/2)^2}{\Gamma(d/2)^2}\int_{r<|z|<|x|}\int_{|z|<|y|} \int_{|w-y|<r}
U^-_x(\d z)
U^+_z(\d y)
\Pi(\d w)
f(z) g(y) |y|^\alpha      h(y+w)\notag\\
&=c_{\alpha,d}
\int_{r<|z|<|x|}\int_{|z|<|y|} \int_{|w-y|<r}
 \frac{||z|^2-|x|^2|^{\alpha/2}||y|^2-|z|^2|^{\alpha/2}}
{|z|^\alpha|z-x|^{d}|z-y|^{d}|w|^{\alpha+d}} \d y\,\d z\, \d w
f(z) g(y)     h(y+w)\\
&=c_{\alpha,d}
\int_{r<|z|<|x|}\int_{|z|<|y|} \int_{|v|<r}
 \frac{||z|^2-|x|^2|^{\alpha/2}||y|^2-|z|^2|^{\alpha/2}}
{|z|^\alpha|z-x|^{d}|z-y|^{d}|v-y|^{\alpha+d}} 
\d y\,\d z\, \d v
f(z) g(y)     h(v)
\end{align*}
where
\[
c_{\alpha,d}
 = \frac{\Gamma(({d+\alpha})/{2})}{|\Gamma(-{\alpha}/{2})|}\frac{\Gamma(d/2)^2}{\pi^{3d/2}\Gamma(\alpha/2)^2}.
\]
 This is equivalent to the statement of part (i) of the theorem.

\medskip

(ii) This is a straightforward application of the Riesz--Bogdan--\.Zak transformation, with computations in the style of those used to prove Lemma \ref{resolventin}. For the sake of brevity, the proof is left as an exercise for the reader.
\end{proof}

\begin{proof}[Proof of Corollary \ref{double1}] As above, we only prove (i) as part (ii) can be derived appealing to the Riesz--Bogdan--\.Zak transformation.

 From \eqref{two-law}, \eqref{U-} and Proposition \ref{Nzeta}, more specifically \eqref{morespecifically},  we have that 
for bounded measurable functions $f,g$ on $\mathbb{R}^d$, 
\begin{align*}
&\mathbb{E}_x[g(X_{{\texttt G}(\tau_1^\oplus)})f(X_{\tau^\oplus_1}); \tau^\oplus_1<\infty]\notag\\
& =  \int_{1<|z|<|x|}U^-_x(\d z)\mathbb{N}_{\arg(z)}(f(|z|{e}^{\epsilon(\zeta)}\Theta^\epsilon(\zeta))\mathbf{1}(|z|{e}^{\epsilon(\zeta)}<1); \zeta<\infty)g(z)\notag\\
&= \frac{\Gamma(d/2)^2\sin(\pi\alpha/2)}{\pi^{d}|\Gamma(-\alpha/2)|\Gamma(\alpha/2)} \int_{1<|z|<|x|}
\int_{|v|<1}
\frac{||z|^2-|x|^2|^{\alpha/2}}{||z|^2-|v|^2|^{\alpha/2}|z-v|^{d}|z-x|^{d}}f(v)g(z)\, \d z\,{\rm d} v.
\end{align*}
This gives the desired result when $r=1$.
As usual, we use scaling to convert the above conclusion to the setting of first passage into a ball of radius $r>0$.
\end{proof}

\begin{proof}[Proof of Corollary \ref{double2}] As with the previous proof, we only deal with (i) and the case that $r=1$ for the same reasons. 
Setting $f\equiv 1$ in \eqref{Levysystem}, we see with the help of Lemma \ref{resolventin} and \eqref{jumpmeasure} that 
\begin{align*}
&\mathbb{E}_x[g(X_{\tau^\oplus_1})h(X_{\tau^\oplus_1}); \tau^\oplus_1<\infty] \\
&=
\mathbb{E}_x\left[\int_0^{\tau^\oplus_1}g(X_t)k(X_t)\right]\\
&=\frac{2^{\alpha}\Gamma(({d+\alpha})/{2})}{\pi^{d/2}|\Gamma(-{\alpha}/{2})|}\int_{|y|>1}g(y) \int_{|y+w|<1}\frac{1}{|w|^{\alpha + d}}\d w \, h(y+w) h_1^\oplus(x,y)\d y \\
&=\frac{2^{\alpha}\Gamma(({d+\alpha})/{2})}{\pi^{d/2}|\Gamma(-{\alpha}/{2})|} \int_{|y|>1} \int_{|v|<1}g(y)h(v)\frac{1}{|v-y|^{\alpha + d}}  h_1^\oplus(x,y)\, \d v \,\d y
\end{align*}
where the function $k(\cdot)$ is as before. The result now follows.
\end{proof}

\section{Deep factorisation of the stable process}\label{deepproof}

The manipulations we have made in  Section \ref{IntegrationN}, in particular in Proposition \ref{Npotential}, are precisely what we need to demonstrate the Wiener--Hopf factorisation. Recall that, for Theorem \ref{WHF}, we defined 
\[
{\bf R}_z[f](\theta) = \mathbf{E}_{0,\theta}\left[\int_0^\infty {\rm e}^{-z \xi_t}f(\Theta_t)\d t\right]
, \qquad \theta\in\mathbb{S}_{d-1}
, z\in\mathbb{C}.
\]
Moreover, define
\[
\hat{\boldsymbol\rho}_z[f](\theta)  = \mathbf{E}_{0,\theta}\left[\int_0^\infty {\rm e}^{- z H^-_t}f(\Theta^-_t)\d t\right] =\int_{|y|<1}|y|^{z}f(\arg(y)) U^-_\theta (\d y)
\]
and 
\[
\boldsymbol{\rho}_z[f](\theta)  = \mathbf{E}_{0,\theta}\left[\int_0^\infty {\rm e}^{- z H^+_t}f(\Theta^+_t)\d t\right] =\int_{|y|>1}|y|^{-z}f(\arg(y)) U^+_\theta (\d y)
\]
for bounded measurable $f: \mathbb{S}_{d-1}\mapsto[0,\infty)$, whenever the integrals make sense. 
We note that the expression for $\boldsymbol{\rho}_z[f](\theta)$ as given in the statement of Theorem \ref{WHF} is clear given \eqref{dualpotential}. Moreover, from e.g. Section 2 of \cite{BGR}, it is known that the free potential measure of a stable process issued from $x\in\mathbb{R}^d$ has density given by 
\[
u(x, y) =\frac{\Gamma((d-\alpha)/2)}{2^\alpha\pi^{d/2}\Gamma(\alpha/2)}|y-x|^{\alpha-d}, \qquad y\in\mathbb{R}^d.
\]
Accordingly, taking account of \eqref{changeofvariableXtoMAP}, it is straightforward to compute
\begin{align*}
{\bf R}_z[f](\theta)& = \mathbf{E}_{0,\theta}\left[\int_0^\infty {\rm e}^{-(z+\alpha) \xi_s}f(\Theta_s){\rm e}^{\alpha \xi_s}\d s\right]\\
&=\mathbb{E}_{\theta}\left[\int_0^\infty |X_t|^{-(\alpha+z)}f(\arg(X_t))\d t\right]\\
&=\int_{\mathbb{R}^d} f(\arg(y)) \frac{u(\theta, y)}{|y-\theta|^{\alpha+z}}\, \d y, \qquad \Re(z)\geq 0,
\end{align*}
where we have used stationary and independent increments in the final equality. Note also that this agrees with the expression for ${\bf R}_z[f](\theta)$ in the statement of Theorem \ref{WHF}.

\begin{proof}[Proof of Theorem \ref{WHF}] 
 From the second and third equalities of  equation \eqref{occupation} (taking $r\to0$) and Proposition \ref{Npotential} gives us 
\begin{align}
{\bf R}_{-{\rm i}\lambda}[f](\theta)&=\int_{|w|<1}U^-_\theta(\d w)\mathbb{N}_{\arg(w)}\left(\int_0^\zeta (|w|{\rm e}^{\epsilon(s)})^{{\rm i }\lambda} f(\Theta^\epsilon(s))\right)\notag\\
&=2^{-\alpha}
\frac{\Gamma((d - \alpha)/2)^2}{\Gamma(d/2)^2}\int_{|w|<1}U^-_\theta(\d w)
\int_{|w|<|y|} f(\arg(y))|y|^{ {\rm i}\lambda} U^+_w(\d y).
\label{WHF1}
\end{align}
Note that, by conditional stationary and independent increments, for any $w\in\mathbb{R}^d\backslash\{0\}$,
\begin{align*}
\int_{|w|<|y|}|y|^{ {\rm i}\lambda}f(\arg(y)) U^+_w (\d y)&=
\mathbf{E}_{\log|w|,\arg(w)}\left[\int_0^\infty {\rm e}^{ {\rm i}\lambda H^+_t}f(\Theta^+_t)\d t\right] \\
&=|w|^{ {\rm i}\lambda}\mathbf{E}_{0,\arg(w)}\left[\int_0^\infty {\rm e}^{ {\rm i}\lambda H^+_t}f(\Theta^+_t)\d t\right] \\
&=|w|^{ {\rm i}\lambda}\int_{1<|y|}|y|^{ {\rm i}\lambda}f(\arg(y)) U^+_{\arg(w)} (\d y).
\end{align*}
Hence back in \eqref{WHF1}, we have 
\begin{align*}
{\bf R}_{-{\rm i}\lambda}[f](\theta)&=
2^{-\alpha}
\frac{\Gamma((d - \alpha)/2)^2}{\Gamma(d/2)^2}\hat{\boldsymbol\rho}_{{\rm i}\lambda}[\boldsymbol{\rho}_{-{\rm i}\lambda}[f]](\theta), \qquad \lambda \in\mathbb{R}.
\end{align*}
Finally we note from \eqref{dualpotential} that, making the change of variables $y= Kw$, so that $\arg(y) = \arg(w)$, $|y| = 1/|w|$ and $\d y = |w|^{-2d}\d w$ and, for  $\theta\in\mathbb{S}_{d-1}$, $|\theta - Kw| = |\theta-w|/|w|$, we have 
\begin{align*}
\hat{\boldsymbol\rho}_{{\rm i}\lambda}[f](\theta) &=\pi^{-d/2}\frac{\Gamma(d/2)^2}{\Gamma((d - \alpha)/2)\Gamma(\alpha/2)} \int_{|y|>1}|y|^{{\rm i}\lambda}f(\arg(y)) 
\frac{||y|^2-1|^{\alpha/2}}
{|y|^\alpha|\theta-y|^{d} }\d y\\
&=\pi^{-d/2}\frac{\Gamma(d/2)^2}{\Gamma((d - \alpha)/2)\Gamma(\alpha/2)} \int_{1<|w|}|w|^{-{\rm i}\lambda+\alpha -d}f(\arg(w)) 
\frac{|1 - |w|^2|^{\alpha/2}}
{|\theta-w|^{d} |w|^\alpha }\d w\\
&=\boldsymbol{\rho}_{{\rm i}\lambda+(\alpha - d)}[f](\theta), \qquad \lambda\in\mathbb{R},
\end{align*}
as required.
\end{proof}

\section{Proof of Theorem \ref{stdist}}\label{8}

Recall from the description of the Riesz--Bogdan--\.Zak transform that $(\xi,\Theta)$ under the change of measure in \eqref{COMMAP} is equal in law to $(-\xi,\Theta)$. Accordingly,
we have for $q>0$, $x\in\mathbb{R}^d\backslash\{0\}$ and bounded measurable $g$ whose support is compactly embedded in the ball of unit radius,
\begin{align*}
&\mathbf{E}_{-\log |x|, \arg(x)}[g({\rm e}^{ -(\overline\xi_{\mathbf{e}_q} -\xi_{\mathbf{e}_q})}\Theta_{\mathbf{e}_q} )]\notag\\
&=
\mathbf{E}_{\log |x|, \arg(x)}\left[ \frac{{\rm e}^{(\alpha-d)\xi_{\mathbf{e}_q}}}{|x|^{\alpha-d}}g({\rm e}^{ -(\xi_{\mathbf{e}_q} -\underline\xi_{\mathbf{e}_q})}\Theta_{\mathbf{e}_q} )\right]\notag\\
&=|x|^{d-\alpha}\mathbf{E}_{\log |x|, \arg(x)}\left[\sum_{\texttt{g}\in G} \mathbf{1}(\zeta_{\texttt{g}'} <\mathbf{e}^{\texttt{g}'}_q,\, \forall G\ni\texttt{g}'<\texttt{g}) {\rm e}^{(\alpha -d){\xi}_{\texttt g}}{\rm e}^{(\alpha -d)\epsilon_\texttt{g}(\mathbf{e}_q^{\texttt{g}})}
g({\rm e^{-\epsilon_\texttt{g}(\mathbf{e}^\texttt{g}_q) }} \Theta^\epsilon_\texttt{g}(\mathbf{e}^\texttt{g}_q))\mathbf{1}(\mathbf{e}_q^\texttt{g} <\zeta_\texttt{g})\right]\notag\\
&=|x|^{d-\alpha}\mathbf{E}_{\log |x|, \arg(x)}\left[\int_0^\infty 
{\rm e}^{-q t} {\rm e}^{(\alpha - d)\xi_t}\mathbb{N}_{\Theta_t}\left({\rm e}^{(\alpha-d)\epsilon(\mathbf{e}_q)}g({\rm e}^{-\epsilon(\mathbf{e}_q)}\Theta(\mathbf{e}_q)); \mathbf{e}_q <\zeta\right)
\d L_t\right]\notag\\
&=|x|^{d-\alpha}\mathbf{E}_{\log |x|, \arg(x)}\left[\int_0^\infty 
{\rm e}^{-q \ell^{-1}_s}{\rm e}^{-(\alpha-d)H^-_s}\mathbb{N}_{\Theta^-_s}\left({\rm e}^{(\alpha-d)\epsilon(\mathbf{e}_q)}g({\rm e}^{-\epsilon(\mathbf{e}_q)}\Theta(\mathbf{e}_q)); \mathbf{e}_q <\zeta\right)
\d s\right],
\end{align*}
where, for each $\texttt{g}\in G$, $\mathbf{e}^{\texttt{g}}_q$ are additional marks on the associated excursion which are independent and exponentially distributed with rate $q$.
Hence, if we define 
\[
U^{(q),-}_x(\d y) = \int_0^\infty \d s\, \mathbf{E}_{\log|x|, \arg(x)}\left[{\rm e}^{-q \ell^{-1}_t}; \, {\rm e}^{-H^-_s}\Theta^-_s\in \d y, \, s<\ell_\infty\right], \qquad |y|<|x|.
\] 
then
\begin{align}
&\mathbf{E}_{-\log |x|, \arg(x)}[g({\rm e}^{ -(\overline\xi_{\mathbf{e}_q} -\xi_{\mathbf{e}_q})}\Theta_{\mathbf{e}_q} )]\notag\\
&= \int_{(0,\infty)}\int_{|y|<|x|}qU^{(q),-}_x(\d y)\frac{|y|^{\alpha -d}}{|x|^{\alpha - d}} \mathbb{N}_{\arg(y)}\left(\int_0^\zeta {\rm e}^{-q t}{\rm e}^{(\alpha-d)\epsilon(t)}g({\rm e}^{-\epsilon(t)}\Theta(t))\d t\right)
\label{limitq}
\end{align}

Recall that $\ell^{-1}$ is a subordinator (without reference to the accompanying modulation $\Theta^+$), suppose we denote its Laplace exponent by $\Lambda^+(q): = - \log\mathbf{E}_{0,\theta}\left[\exp\{-q L^{-1}_1\}\right]$, $q\geq 0$, where $\theta\in\mathbb{S}_{d-1}$ is unimportant in the definition.
Appealing again to the Riesz--Bogdan--\.Zak transform again, we also note that for a bounded and measurable function $h$ on $\mathbb{S}_{d-1}$, using obvious notation
\begin{align}
\int_{|y|<|x|}\frac{|y|^{\alpha -d}}{|x|^{\alpha - d}}qU^{(q),-}_x(\d y) h(\arg(y))&= q\int_0^\infty \d s\, \mathbf{E}_{-\log|x|, \arg(x)}\left[{\rm e}^{-q L^{-1}_t} h( \Theta^+_s)\right]\notag\\
&=\frac{q}{\Lambda^+(q)} \int_0^\infty \d s\,\Lambda^+(q){\rm e}^{-\Lambda^+(q)s} \mathbf{E}^{(q)}_{-\log|x|, \arg(x)}\left[ h( \Theta^+_s)\right]\notag\\
&=\frac{q}{\Lambda^+(q)} \mathbf{E}^{(q)}_{-\log|x|, \arg(x)}\left[ h\left(\Theta^+_{\mathbf{e}_{\Lambda^+(q)}} \right)\right]
\label{takeqto0}
\end{align}
where 
$
\mathbf{P}^{(q)} _{-\log |x|, \arg(x)}
$
appears as the result of a change of measure with martingale density $\exp\{-qL^{-1}_s +\Lambda^+(q)s\}$, $s\geq0$, and $\Lambda^+(q)$ is the Laplace exponent of the subordinator $L^{-1}$ and $\mathbf{e}_{\Lambda^+(q)}$ is an independent exponential random variable with parameter $\Lambda^+(q)$.

Next, we want to take $q\downarrow0$ in \eqref{limitq}. To this end, we start by remarking that, as $L$ is a local time for the L\'evy process $\xi$ (without reference to its modulation), it is known from classical Wiener--Hopf factorisation theory that, up to a multiplicative constant, $c>0$, which depends on the normalisation of the local time $L$,   $q = c\Lambda^+(q)\Lambda^-(q)$, where $\Lambda^-(q)$ is the Laplace exponent of the local time at the infimum $\ell$; see for example equation (3) in Chapter VI of \cite{bertoin}. 
On account of the fact that $X$ is transient, we know that $\ell_\infty$ is exponentially distributed and the reader may recall that we earlier normalised our choice of $\ell$ such that its rate, $\Lambda^-(0)=1$. This implies, in turn, that $\lim_{q\downarrow0}q/\Lambda^+(q) =c$.

Appealing to isotropy, the recurrence of $\{0\}\times\mathbb{S}_{d-1}$ for $(\overline\xi - \xi, \Theta)$ and weak convergence back in \eqref{takeqto0} as we take the limit with $q\downarrow0$, to find that 
\[
\lim_{q\to0}\int_{|y|<|x|}\frac{|y|^{\alpha -d}}{|x|^{\alpha - d}}qU^{(q),-}_x(\d y) h(\arg(y))=c\int_{\mathbb{S}_{d-1}}\sigma_1(\d \phi)h(\phi),
\]
where we recall that $\sigma_1(\d \phi)$ is the surface measure on $\mathbb{S}_{d-1}$ normalised to have unit mass.
Hence, back in \eqref{limitq} we have with the help of Proposition \ref{Npotential} and \eqref{dualpotential},
\begin{align}
&\lim_{q\downarrow0}\mathbf{E}_{-\log |x|, \arg(x)}[g({\rm e}^{ -(\overline\xi_{\mathbf{e}_q} -\xi_{\mathbf{e}_q})}\Theta_{\mathbf{e}_q} )] \notag\\
&= \int_{\mathbb{S}_{d-1}}\sigma_1(\d \phi) \mathbb{N}_{\phi}\left(\int_0^\zeta {\rm e}^{(\alpha-d)\epsilon(t)}g({\rm e}^{-\epsilon(t)}\Theta(t))\d t\right)\notag\\
&=c\pi^{-d/2}2^{-\alpha}\frac{\Gamma((d - \alpha)/2)}{\Gamma(\alpha/2)}
\int_{\mathbb{S}_{d-1}} \sigma_1(\d\phi)\int_{1<|z|} g(Kz) 
\frac{||z|^2-1|^{\alpha/2}}
{|z|^d|\phi-z|^{d} }\d z,
\label{almostthere}
\end{align}
where we recall that $Kz = z/|z|^2$.
Finally, we note that, using the Lamperti--Kiu transform and  \eqref{changeofvariableXtoMAP}, for bounded measurable $f$ and compactly embedded in $\mathbb{B}_d$,
\begin{align*}
f\left(
{X_t}/{M_t}
 \right)\d t
&= f\left(
{\rm e}^{-(\overline{\xi}_s-\xi_s)}\Theta_s
 \right){\rm e}^{\alpha \xi_s}\d s
\end{align*}
where $s= \varphi (t)$, suggesting that, for $y\in\mathbb{R}^d\backslash\{0\}$,  
\[
\lim_{t\to\infty}\mathbb{E}_y[ f\left(
{X_t}/{M_t}
 \right)] = \lim_{s\to\infty}\mathbf{E}_{\log|y|, \arg(y)}
 \left[f\left(
{\rm e}^{-(\overline{\xi}_s-\xi_s)}\Theta_s
 \right){\rm e}^{\alpha \xi_s}\right].
\]
In fact, we can justify this rigorously appealing to the  discussion at the bottom of p240 of \cite{Walsh}. Hence, putting this  together with \eqref{almostthere}, for $f$ and $x$ as before,  we conclude that,  
\begin{align}
\lim_{t\to\infty}\mathbb{E}_{Kx}[ f\left(
{X_t}/{M_t}
 \right)]&=\lim_{q\downarrow0}\mathbf{E}_{-\log |x|, \arg(x)}[f({\rm e}^{ -(\overline\xi_{\mathbf{e}_q} -\xi_{\mathbf{e}_q})}\Theta_{\mathbf{e}_q} ){\rm e}^{\alpha \xi_{\mathbf{e}_q} }] \notag\\
 &=c\pi^{-d/2}2^{-\alpha}\frac{\Gamma((d - \alpha)/2)}{\Gamma(\alpha/2)}
\int_{\mathbb{S}_{d-1}} \sigma_1(\d\phi)\int_{1<|z|} f(Kz) 
\frac{|Kz|^\alpha ||z|^2-1|^{\alpha/2}}
{|z|^d|\phi-z|^{d} }\d z\notag\\
&=c\pi^{-d/2}2^{-\alpha}\frac{\Gamma((d - \alpha)/2)}{\Gamma(\alpha/2)}
\int_{\mathbb{S}_{d-1}} \sigma_1(\d\phi)\int_{|w|<1} f(w) 
\frac{|1-|w|^2|^{\alpha/2}  }
{|\phi-w|^{d} }{\d w}
\label{c}
\end{align}
where we changed variables $w = Kz$, or equivalently $z = Kw$, and we  used \eqref{Kz-Kz}, that $|w| = 1/|z|$ and that $\d z = \d w/ |w|^{2d} $. 

In order to pin down the constant $c$, we need to ensure that, when $f\equiv1$, the integral on the right-hand side of \eqref{c} is identically equal to 1.
To do this, we recall a classical Poisson potential formula (see for example Theorem 4.3.1 in \cite{PortStone})
\begin{equation}
( 1-|w|^2)^{-1}= \int_{\mathbb{S}_{d-1}}|\phi-w|^{-d}
\sigma_1(\d \phi)\qquad |w|<1.
\label{provesNewton}
\end{equation}
Writing $\sigma_r(\d \theta)$, $\theta\in r\mathbb{S}_{d-1}$ for the uniform surface measure on $r\mathbb{S}_{d-1}$ normalised to have total mass equal to one, 
it follows that 
\begin{align*}
\int_{\mathbb{S}_{d-1}} \sigma_1(\d\phi)\int_{|w|<1} 
\frac{|1-|w|^2|^{\alpha/2}  }
{|\phi-w|^{d} }{\d w}
&=\int_{|w|<1}|1-|w|^2|^{\frac{\alpha}{2}-1}\d w\\
&= \frac{2\pi^{d/2}}{\Gamma(d/2)}\int_0^1r^{d-1}\d r\int_{r\mathbb{S}_{d-1}}\sigma_r(\d \theta)(1-r^2)^{\frac{\alpha}{2}-1}\\
&= \frac{\pi^{d/2}}{\Gamma(d/2)}\int_0^1 y^{\frac{d}{2}-1} (1-y)^{\frac{\alpha}{2}-1}\d y\\
&=\pi^{d/2}\frac{\Gamma(\alpha/2)}{\Gamma((d+\alpha)/2)}.
\end{align*}
This forces us to take 
\[
c = 2^{\alpha}\frac{\Gamma((d+\alpha)/2)}{\Gamma((d - \alpha)/2)}
\]
and so,we have 
\begin{align*}
\lim_{t\to\infty}\mathbb{E}_{Kx}[ f\left(
{X_t}/{M_t}
 \right)]&=
\pi^{-d/2}\frac{\Gamma((d+\alpha)/2)}{\Gamma(\alpha/2)}
\int_{\mathbb{S}_{d-1}} \sigma_1(\d\phi)\int_{|w|<1} f(w) 
\frac{|1-|w|^2|^{\alpha/2}  }
{|\phi-w|^{d} }{\d w}.
 \end{align*}
as required. \hfill$\square$

\section*{Acknowledgements} Both authors would like to thank Ron Doney who pointed out the distributional interpretations in Remarks \ref{rem1} and \ref{stdist}.

\bibliography{references}{}
\bibliographystyle{plain}

\end{document}